\documentclass[
11pt
]{amsart}
\usepackage{amsmath,amsfonts,amsthm,amssymb,dsfont}
\usepackage[alphabetic]{amsrefs}
\usepackage[OT4]{fontenc}
\usepackage{enumerate}
\usepackage{eucal}
\usepackage{enumerate, xspace, ifthen}
\usepackage{graphicx}

\long\def\symbolfootnote[#1]#2{\begingroup%
\def\thefootnote{\fnsymbol{footnote}}\footnote[#1]{#2}\endgroup}

\def\ds{\rule{0pt}{1.5ex}}

\newtheorem{theorem}{Theorem}[section]
\newtheorem{lemma}[theorem]{Lemma}
\newtheorem{lem}[theorem]{Lemma}
\newtheorem{thm}[theorem]{Theorem}
\newtheorem{cri}[theorem]{Criterion}

\newtheorem{prop}[theorem]{Proposition}

\newtheorem{cor}[theorem]{Corollary}

\theoremstyle{definition}
\newtheorem{rem}[theorem]{Remark}
\newtheorem{defin}[theorem]{Definition}
\newtheorem{constr}[theorem]{Construction}

\newcommand{\R}{\mathbb{R}}
\newcommand{\Z}{\mathbb{Z}}
\renewcommand{\H}{\mathbb{H}}
\newcommand{\p}{\textup{\textsf{p}}}
\newcommand{\hb}{\textsf{h}}
\newcommand{\gb}{\textsf{g}}

\newcommand{\homology}{\ensuremath{{\sf{H}}}}
\newcommand{\systole}[1]{\ensuremath{\vert\!\vert #1 \vert\!\vert}}
\newcommand{\nclose}[1]{\ensuremath{\langle\!\langle#1\rangle\!\rangle}}

\begin{document}

\title{Mixed $3$--manifolds are virtually special}
\author[P.~Przytycki]{Piotr Przytycki$^\dag$}
\address{Inst. of Math., Polish Academy of Sciences\\
 \'Sniadeckich 8, 00-956 Warsaw, Poland}
\email{pprzytyc@mimuw.edu.pl}
\thanks{$\dag$ Partially supported by MNiSW grant N201 012 32/0718, the Foundation for Polish Science, and National Science Centre DEC-2012/06/A/ST1/00259.}
\author[D.~T.~Wise]{Daniel T. Wise$^\ddag$}
           \address{Math. \& Stats.\\
                    McGill University \\
                    Montreal, Quebec, Canada H3A 2K6 }
           \email{wise@math.mcgill.ca}
\thanks{$\ddag$ Supported by NSERC}

\maketitle

\begin{abstract}
\noindent
Let $M$ be a compact oriented irreducible $3$--manifold which is neither a graph manifold nor a hyperbolic manifold. We prove that $\pi_1M$ is virtually special.
\end{abstract}

\section{Introduction}
A compact connected oriented irreducible $3$--manifold with arbitrary, possibly empty boundary is \emph{mixed} if it is not hyperbolic and not a graph manifold.
A group is \emph{special} if it is a subgroup of a right-angled Artin group. Our main result is the following.

\begin{thm}
\label{thm:main}
Let $M$ be a mixed $3$--manifold. Then $\pi_1M$ is virtually special.
\end{thm}

\begin{cor}
\label{cor:linear}
The fundamental group of a mixed $3$--manifold is linear over $\Z$.
\end{cor}

As explained below, Theorem~\ref{thm:main} has the following consequence.

\begin{cor}
\label{cor:fiber}
A mixed $3$--manifold with possibly empty toroidal boundary virtually fibers.
\end{cor}

An alternative definition of a special group
is the following. A nonpositively curved cube complex $X$ is \emph{special} if its immersed hyperplanes do not self-intersect, are two-sided, do not directly self-osculate or interosculate (see Definition~\ref{def:special}).
A group $G$ is (\emph{compact}) \emph{special} if it is the fundamental group of a (compact) special cube complex $X$. Then $G$ is a subgroup of a possibly infinitely generated right-angled Artin group \cite[Thm 4.2]{HW}. Conversely, a subgroup $G$ of a right-angled Artin group is the fundamental group of the corresponding cover $X$ of the Salvetti complex, which is special. Note that if the fundamental group $G$ of a special cube complex $X$ is finitely generated, then a minimal locally convex subcomplex $X'\subset X$ containing a $\pi_1$--surjective finite graph in $X^1$ has finitely many hyperplanes and is special, so that $G$ embeds in a finitely generated right-angled Artin group.

Special groups are residually finite. Moreover, assuming that $X$ has finitely many hyperplanes, the stabilizer in $G$ of any hyperplane in the universal cover $\widetilde{X}$ of $X$ is separable (see Corollary~\ref{cor:special->separable}). For $3$--manifold groups, separability of a subgroup corresponding to an immersed incompressible surface implies that in some finite cover of the manifold the surface lifts to an embedding. There are immersed incompressible surfaces in graph manifolds that do not lift to embeddings in a finite cover \cite{RW}.

There are a variety of groups with the property that every finitely generated subgroup is separable --- for instance, this was shown for free groups by M.\ Hall and for surface groups by Scott. A compact $3$--manifold is \emph{hyperbolic} if its interior is homeomorphic to a quotient of $\H^3$ (equivalently to the quotient of the interior of the convex hull of the limit set) by a geometrically finite Kleinian group. It was recently proved that hyperbolic $3$--manifolds with an embedded geometrically finite incompressible surface have fundamental groups that are virtually compact special \cite[Thm 14.29]{Hier}.
This implies separability for all geometrically finite subgroups \cite[Thm 16.23]{Hier}. By Tameness \cite{Atam,CG} and Covering \cite{Thu,Can} Theorems all other finitely generated subgroups correspond to virtual fibers and hence they are separable as well.
Very recently, Agol, Groves and Manning \cite[Thm 1.1]{Ahak}
building on \cite{Hier} proved that the fundamental group of every closed hyperbolic $3$--manifold is virtually compact special and hence all its finitely generated subgroups are separable. For more details, see the survey article \cite{AFW}.

Another striking consequence of virtual specialness is virtual fibering. Since special groups are subgroups of right-angled Artin groups, they are subgroups of right-angled Coxeter groups as well \cite{HsuW,DJ}. Agol proved that such groups are virtually residually finite rationally solvable (RFRS) \cite[Thm 2.2]{A}. Then he proved that if the fundamental group of a compact connected oriented irreducible $3$--manifold with toroidal boundary is RFRS, then it virtually fibers \cite[Thm 5.1]{A}. In view of these results, every hyperbolic manifold with toroidal boundary virtually fibers \cite[Thm 9.2]{Ahak}. Similarly, our Theorem~\ref{thm:main} yields Corollary~\ref{cor:fiber}.

Liu proved that an aspherical graph manifold has virtually special fundamental group if and only if it admits a nonpositively curved Riemannian metric \cite[Thm 1.1]{Liu}. Independently, and with an eye
towards the results presented here, we proved virtual specialness for graph manifolds with nonempty boundary \cite[Cor 1.3]{PW}. Note that graph manifolds with nonempty boundary carry a nonpositively curved metric by \cite[Thm 3.2]{Leeb}. Our Theorem~\ref{thm:main} thus resolves the question of virtual specialness for arbitrary compact $3$--manifold groups.

\begin{cor}
\label{cor:Liu}
A compact aspherical $3$--manifold has virtually special fundamental group if and only if it admits a Riemannian metric of nonpositive curvature.
\end{cor}

Corollary~\ref{cor:Liu} was conjectured by Liu \cite[Conj 1.3]{Liu}. As discussed above, he proved the conjecture for graph manifolds while for hyperbolic manifolds this follows from \cite{Hier} and \cite{Ahak}. All mixed manifolds admit a metric of nonpositive curvature, essentially due to \cite[Thm 3.3]{Leeb}, as showed in \cite[Thm 4.3]{Bri}. Hence Theorem~\ref{thm:main} resolves Liu's conjecture in the remaining mixed case.
However, the equivalence in Corollary~\ref{cor:Liu} appears to be more circumstantial than a consequence of an intrinsic relationship between nonpositive curvature and virtual specialness: all manifolds in question except for certain particular closed graph manifolds have both of these features.

As a consequence of virtual specialness of mixed manifolds (Theorem~\ref{thm:main}), hyperbolic manifolds with nonempty boundary \cite[Thm 14.29]{Hier}, and graph manifolds with nonempty boundary \cite[Cor 1.3]{PW} we have the following.

\begin{cor}
The fundamental group of any knot complement in $S^3$ has a faithful representation in $\mathrm{SL}(n,\Z)$ for some $n$.
\end{cor}

Note that existence of a non-abelian representation of any nontrivial knot complement group into $\mathrm{SU}(2)$ is a well-known result of Kronheimer--Mrowka \cite{KMr}.

\medskip

\noindent \textbf{Organization.}
As explained in Section~\ref{sec:div}, the proof of Theorem~\ref{thm:main} is divided into two steps.
The first step is Theorem~\ref{thm:cubulation} (Cubulation), which roughly states that in any mixed manifold there is a collection of surfaces sufficient for cubulation. In Section~\ref{sec:graph} we review the construction of surfaces in graph manifolds with boundary. We discuss surfaces in hyperbolic manifolds with boundary in Section~\ref{sec:hyperbolic}. We prove Theorem~\ref{thm:cubulation} in Section~\ref{sec:cubul} by combining the surfaces from graph manifold blocks and hyperbolic blocks.

The second step is Theorem~\ref{thm:specialization} (Specialization), which provides the virtual specialness of the nonpositively curved cube complex produced in the first step. In Section~\ref{sec:separability} we extend some separability results for special cube complexes to non-compact setting.
We apply them in Section~\ref{sec:small_cancellation} to obtain cubical small cancellation results for non-compact special cube complexes. This allows us to prove Theorem~\ref{thm:specialization} in Section~\ref{sec:special}.

\medskip

\noindent \textbf{Ingredients in the proof of Theorem~\ref{thm:main}:}

\begin{itemize}
\item
Canonical completion and retraction (Theorem~\ref{thm:canonical_completion_and_retraction}) for special cube complexes~\cite{HW2}.
\item
Criterion~\ref{thm:HWcriterion} for virtual specialness~\cite{HW2}.
\item
Gitik--Minasyan double
quasiconvex coset separability~\cite{Min}.
\item
Criterion for relative quasiconvexity~\cite{BW}.
\item
Combination Theorem~\ref{thm:Eduardo} for relatively quasiconvex groups~\cite{MP}.
\item
Relative cocompactness of cubulations of relatively hyperbolic groups~\cite{HruW}.
\item
Criterion~\ref{thm:WallNbd} for WallNbd-WallNbd Separation \cite{HruW}.
\item
Proposition~\ref{prop:existence_graph} which constructs virtually embedded surfaces in graph manifolds with boundary~\cite{PW}.
\item
Separability and double separability of embedded surfaces in graph manifolds~\cite{PW}.
\item
Special Quotient Theorem~\ref{thm:relative special qoutient} for groups hyperbolic relative to free-abelian subgroups~\cite{Hier}.
\item
Theorem~\ref{thm:hyperbolic->special} on virtual specialness of hyperbolic manifolds with nonempty boundary~\cite{Hier}.
\item
Main Theorem~\ref{thm:ladder} of cubical small cancellation~\cite{Hier}.
\end{itemize}

\noindent
\textbf{Acknowledgement.} We thank Stefan Friedl for his remarks and corrections. We also thank the referee for detailed comments that helped us clarify the proof.

\section{Technical reduction to two steps}
\label{sec:div}

Let $M$ be a compact connected oriented irreducible $3$--manifold. By passing to a double cover, we can also assume that $M$ has no $\pi_1$--injective Klein bottles. Moreover, assume that $M$ is not a Sol or Nil manifold. Up to isotopy, $M$ then has a unique minimal collection of incompressible tori not parallel to $\partial M$, called \emph{JSJ tori}, such that the complementary components called \emph{blocks} are either algebraically atoroidal or else Seifert fibered \cite[Thm 3.4]{Bon}. We say that $M$ is \emph{mixed} if it has at least one JSJ torus and one atoroidal block. (Equivalently, by Perelman's geometrization, $M$ is not hyperbolic and not a graph manifold.) By Thurston's hyperbolization all atoroidal blocks are hyperbolic and we will denote them by $M^{\hb}_k$. The JSJ tori adjacent to at least one hyperbolic block are \emph{transitional}. The complementary components of the union of the hyperbolic blocks are graph manifolds with boundary and will be called \emph{graph manifold blocks} and denoted by $M^{\gb}_i$. Up to a diffeomorphism isotopic to the identity, each of their Seifert fibered blocks admits a unique Seifert fibration that we fix. If a transitional torus is adjacent on both of its sides to hyperbolic blocks, we replace it by two parallel tori (also called JSJ, and transitional) and add the product region $T\times I$ bounded by them as a graph manifold block to the family $\{M^{\gb}_i\}$. Similarly, for a boundary torus of $M$ adjacent to a hyperbolic block, we introduce its parallel copy in $M$ (called JSJ, and transitional) and add the product region to $\{M^{\gb}_i\}$.
These $M^{\gb}_i=T\times I$ will be called \emph{thin}. We will later fix one of many Seifert fibrations on thin $M^{\gb}_i$.

Unless stated otherwise all surfaces are embedded or immersed \emph{properly}.
Let $S\rightarrow M$ be an immersed surface in a $3$--manifold. Let $\widehat{M}\rightarrow M$ be a covering map. A map $\widehat{S}\rightarrow \widehat{M}$ that covers $S\rightarrow M$ and does not factor through another such map is its \emph{elevation} (it is a \emph{lift} when $\widehat{S}=S$). A connected oriented surface $S\rightarrow M$ that is not a sphere is \emph{immersed incompressible} if it is $\pi_1$--injective and its elevation to the universal cover $\widetilde{M}$ of $M$ is an embedding. The surface $S$ is \emph{virtually embedded} if there is a finite cover $\widehat{M}$ of $M$ with an embedded elevation of $S$. Given a block $B$ and an immersed surface $\phi\colon S\rightarrow M$, a \emph{piece} of $S$ in $B$ is the restriction of $\phi$ to a component of $S\cap \phi^{-1}(B)$.

The elevations of JSJ tori, boundary tori, and transitional tori of $M$ to the universal cover $\widetilde{M}$ are called \emph{JSJ planes}, \emph{boundary planes}, and \emph{transitional planes}, and we keep the name \emph{blocks} (\emph{hyperbolic, graph manifold}, or \emph{Seifert fibered}) for the elevations of blocks of $M$.
We warn that this terminology refers to ``graph manifold blocks'' in $\widetilde{M}$ even though they are not compact.
Having specified a block $\widetilde{M}_o$ of $\widetilde{M}$ and a surface $\widetilde{S}_o\subset\widetilde{M}_o$ we denote by $\mathcal T(\widetilde{S}_o)$ the set of JSJ and boundary planes in $\partial \widetilde{M}_o$ intersecting $\widetilde{S}_o$.

An \emph{axis} for an element $g\in \pi_1M$ acting on $\widetilde{M}$ is a copy of $\R$ in $\widetilde{M}$ on which $g$ acts by nontrivial translation.
A \emph{cut-surface} for $g\in \pi_1M$ is an immersed incompressible surface $S\rightarrow M$ covered by $\widetilde{S}\subset \widetilde{M}$ such that
there is an axis $\R$ for $g$ satisfying $\widetilde{S}\cap \R=\{0\}$, where the intersection is transverse.

\begin{thm}[Cubulation]
\label{thm:cubulation}
Let $M$ be a mixed $3$--manifold. There is a finite
family of immersed incompressible surfaces $\mathcal{S}$ in $M$, in general position, and such that:
\begin{enumerate}
\item
For each element of $\pi_1M$ there is a cut-surface in $\mathcal{S}$.
\item
All JSJ tori belong to $\mathcal{S}$.
\item
Each piece of $S$ in $M^{\gb}_i$ is virtually embedded in $M^{\gb}_i$ for each $S\in\mathcal{S}$.
\item
Each piece of $S$ in $M^{\hb}_k$ is geometrically finite for each $S\in\mathcal{S}$.
\item
The family $\mathcal{S}$ satisfies the following Strong Separation property.
\end{enumerate}
\end{thm}

To make sense of \emph{sufficiently far} below we fix a Riemannian metric on $M$ and lift it to the universal cover $\widetilde{M}$. Note, however, that satisfying Strong Separation does not depend on the choice of this metric.

\begin{defin}
A family $\mathcal{S}$ of surfaces in $M$ satisfies the \emph{Strong Separation} property if for the family
$\widetilde{\mathcal{S}}$ of elevations to $\widetilde{M}$ of the surfaces in $\mathcal{S}$ the following hold.

\begin{enumerate}[(a)]
\item
For any $\widetilde{S}, \widetilde{S}'\in \widetilde{\mathcal{S}}$ intersecting a block $\widetilde{M}^{\hb}_k$ covering $M^{\hb}_k$, if $\widetilde{S}'\cap \widetilde{M}^{\hb}_k$ and $\widetilde{S}\cap \widetilde{M}^{\hb}_k$ are sufficiently far and
$\mathcal T(\widetilde{S}'\cap \widetilde{M}^{\hb}_k)\cap \mathcal T(\widetilde{S}\cap \widetilde{M}^{\hb}_k)=\emptyset$,
then a surface from $\widetilde{\mathcal{S}}$ separates $\widetilde{S}'$ from $\widetilde{S}$.
\item
For any $\widetilde{S},\widetilde{S}'\in \widetilde{\mathcal{S}}$ intersecting a block $\widetilde{M}^{\gb}_i$ covering $M^{\gb}_i$,
if $\widetilde{S}'\cap\widetilde{M}^{\gb}_i$ and $\widetilde{S}\cap \widetilde{M}^{\gb}_i$ are sufficiently far,
then a surface from $\widetilde{\mathcal{S}}$ separates $\widetilde{S}'$ from $\widetilde{S}$.
\end{enumerate}
\end{defin}

We consider the \emph{dual $\mathrm{CAT(0)}$ cube complex} $\widetilde{X}$ associated to $\mathcal{S}$ by
Sageev's construction. Each $\widetilde{S}\in\widetilde{\mathcal{S}}$ cuts $\widetilde{M}$ into two closed halfspaces $U,V$ and the collection of pairs $\{U,V\}$ endows $\widetilde{M}$ with a Haglund--Paulin wallspace structure (we follow the treatment of these ideas in \cite[\S 2.1]{HruW} where $U\cap V$ is allowed to be nonempty). The group $G=\pi_1M$ acting on $\widetilde{M}$ preserves this structure and hence it acts on the associated dual CAT(0) cube complex $\widetilde{X}$. The stabilizer in $G$ of a hyperplane in $\widetilde{X}$ coincides with a conjugate of $\pi_1S$ for an appropriate $S\in \mathcal{S}$ by general position. Note that if there is a cut-surface $S\in\mathcal{S}$ for $g\in G$, then $g$ acts freely on $\widetilde{X}$ \cite[Lem 7.16]{Riches}.

\smallskip
If a group $G$ acting freely on a CAT(0) cube complex $\widetilde{X}$ has a finite index subgroup $G'$ such that $G'\backslash \widetilde{X}$ is special, then we say that the action of $G$ on $\widetilde{X}$ is \emph{virtually special}. This coincides with the definition used in \cite{HW2}, by the freeness of the action and \cite[Thm~3.5 and Rem~3.6]{HW2}. We prove Theorem~\ref{thm:main} using the following criterion for virtual specialness. Disjoint hyperplanes \emph{osculate} if they have dual edges sharing an endpoint.

\begin{cri}
\label{thm:HWcriterion}
Let $G$ act freely on a $\mathrm{CAT(0)}$ cube complex $\widetilde{X}$. Suppose that:
\begin{enumerate}
\item
there are finitely many $G$ orbits of hyperplanes in $\widetilde{X}$,
\item
for each hyperplane $\widetilde{A}\subset\widetilde{X}$, there are finitely many $\mathrm{Stab}(\widetilde{A})$ orbits of hyperplanes that intersect $\widetilde{A}$,
\item
for each hyperplane $\widetilde{A}\subset\widetilde{X}$, there are finitely many $\mathrm{Stab}(\widetilde{A})$ orbits of hyperplanes that osculate with $\widetilde{A}$,
\item
for each hyperplane $\widetilde{A}\subset\widetilde{X}$, the subgroup $\mathrm{Stab}(\widetilde{A})\subset G$ is separable, and
\item
for each pair of intersecting hyperplanes $\widetilde{A},\widetilde{B}\subset\widetilde{X}$, the double coset $\mathrm{Stab}(\widetilde{A})\mathrm{Stab}(\widetilde{B})\subset G$ is separable.
\end{enumerate}Then the action of $G$ on $\widetilde{X}$ is virtually special.
\end{cri}

Criterion~\ref{thm:HWcriterion} follows directly from \cite[Thm~4.1]{HW2}, since in Conditions (4)~and~(5) we require $\mathrm{Stab}(\widetilde{A})$ and $\mathrm{Stab}(\widetilde{A})\mathrm{Stab}(\widetilde{B})$ to be closed in the profinite topology on $G$, and not only to have have closures disjoint from certain specified sets as was required in \cite[Thm~4.1]{HW2}.

\medskip

For each $M^{\gb}_i$ we choose one conjugate $P_i$ of $\pi_1M^{\gb}_i$ in $G=\pi_1M$. Then $G$ is hyperbolic relative to $\{P_i\}$ (see e.g.\ \cite{BW}) and we can discuss quasiconvexity of its subgroups relative to $\{P_i\}$ (see e.g.\ \cite[Def 2.1]{BW}). For each $S\in \mathcal{S}$, Theorem~\ref{thm:cubulation}(4) implies that $\pi_1S$ is quasiconvex in $G$ relative to $\{P_i\}$, by \cite[Thm 4.16]{BW}.

Let $\widetilde{M}^{\gb}_i\subset \widetilde{M}$ be the elevation of $M^{\gb}_i$ stabilized by $P_i$. We describe a convex $P_i$--invariant subcomplex $\widetilde{Y}_i\subset\widetilde{X}$ determined by $\widetilde{M}^{\gb}_i$.
Let $\mathcal{U}_i$ be the family of halfspaces $U$ in the wallspace $\widetilde{M}$ for which there is some $R>0$ with $\mathrm{diam}(U\cap N_R(\widetilde{M}^{\gb}_i))=\infty$, where $N_R$ denotes the $R$--neighborhood. Note that $\mathcal{U}_i$ does not depend on the fixed Riemannian metric on $M$. Let $\widetilde{Y}_i\subset\widetilde{X}$ be the subcomplex consisting of cubes spanned by the vertices whose halfspaces are all in $\mathcal{U}_i$.

By~\cite[Thm 7.12]{HruW} the group $G$
\emph{acts cocompactly on $\widetilde{X}$ relative to $\{\widetilde{Y}_i\}$} in the following sense: there exists a compact subcomplex $K\subset\widetilde{X}$ such that:
\begin{itemize}
\item
$\widetilde{X}=GK\cup\bigcup_iG\widetilde{Y}_i$,
\item
$g\widetilde{Y}_i\cap \widetilde{Y}_j\subset GK$ unless $j=i$ and $g\in P_i$, and
\item
$P_i$ acts cocompactly on $\widetilde{Y}_i\cap GK$.
\end{itemize}

Because $G$ acts freely on $\widetilde{X}$, by \cite[Prop 8.1(1)]{HruW} each $\widetilde{Y}_i$ is \emph{superconvex}, in the sense that there is a uniform bound on the diameter of a rectangle $[-d,d]\times I$ isometrically embedded in $\widetilde{X}$ on $1$--skeleton with $[-d,d]\times\{0\}\subset\widetilde{Y}_i$ and $[-d,d]\times\{1\}$ outside $\widetilde{Y}_i$.

Observe that $G=\pi_1M$ splits as a graph of groups with transitional tori groups as edge groups. The group $G$ is hyperbolic relative to the vertex groups $P_i=\pi_1M_i^{\gb}$. We now explain that to prove Theorem~\ref{thm:main} it suffices to complement Theorem~\ref{thm:cubulation} with the following.

\begin{thm}[Specialization]
\label{thm:specialization}
Let $G$ be the fundamental group of a graph of groups with free-abelian edge groups. Suppose that $G$ is hyperbolic relative to some collection of the vertex groups $\{P_i\}$. Suppose that $G$ acts cocompactly on a CAT(0) cube complex $\widetilde{X}$ relative to superconvex $\{\widetilde{Y}_i\}$. Suppose also that:

\begin{enumerate}[(i)]
\item
the action of $G$ on $\widetilde{X}$ is free and satisfies finiteness conditions (1)--(3) of Criterion~\ref{thm:HWcriterion},
\item
for any finite index subgroup $E^\circ$ of an edge group $E\subset P_i$, there is a finite index subgroup $P_i'\subset P_i$ with $P_i'\cap E\ \subset E^\circ$,
\item
the action of each $P_i$ on $\widetilde{Y}_i$ is virtually special, with finitely many orbits of codim--$2$--hyperplanes,
\item
each non-parabolic vertex group is virtually compact special.
\end{enumerate}
Then the action of $G$ is virtually special.
\end{thm}

A \emph{codim--$2$--hyperplane} in a CAT(0) cube complex is the intersection of a pair of intersecting hyperplanes.

\smallskip

We now derive the hypothesis of Theorem~\ref{thm:specialization} from the conclusion of Theorem~\ref{thm:cubulation}. By Theorem~\ref{thm:cubulation}(1), the action of $\pi_1M$ on $\widetilde{X}$ is free. Moreover, since the family $\mathcal{S}$ is finite, Condition~(1) of Criterion~\ref{thm:HWcriterion} is satisfied, and since $\mathcal{S}$ is in general position, we have Condition~(2). We now deduce Condition~(3). Disjoint hyperplanes in a CAT(0) cube complex osculate (i.e.\ have dual edges sharing an endpoint) if and only if they are not separated by another hyperplane. (The if part follows from the observation that a hyperplane dual to an edge of a shortest path between the carriers of disjoint hyperplanes separates them.) Similarly, we say that two disjoint surfaces $\widetilde{S},\widetilde{S}'\in\widetilde{\mathcal{S}}$ \emph{osculate} if there is no surface in $\widetilde{\mathcal{S}}$ separating $\widetilde{S}'$ from $\widetilde{S}$. Hence osculating hyperplanes in $\widetilde{X}$ correspond to osculating $\widetilde{S},\widetilde{S}'\in \widetilde{\mathcal{S}}$. We need to show that there are finitely many $\mathrm{Stab}(\widetilde{S})$ orbits of surfaces in $\widetilde{\mathcal{S}}$ osculating with $\widetilde{S}$. Note that if $\widetilde{S}'$ osculates with $\widetilde{S}$, then it must intersect one of the finitely many $\mathrm{Stab}(\widetilde{S})$ orbits of graph manifold and hyperbolic blocks intersected by $\widetilde{S}$, since otherwise it would be separated from $\widetilde{S}$ by a transitional plane $\widetilde{T}$. But $\widetilde{T}\in \widetilde{\mathcal{S}}$ by Theorem~\ref{thm:cubulation}(2), so $\widetilde{S}$ and $\widetilde{S}'$ would not osculate. If both $\widetilde{S}$ and $\widetilde{S}'$ intersect the same block $\widetilde{M}^{\gb}_i$, then by Strong Separation (b) of Theorem~\ref{thm:cubulation}(5) they are at bounded distance, and so there are finitely many $\mathrm{Stab}(\widetilde{S})$ orbits. If $\widetilde{S}$ and $\widetilde{S}'$ do not intersect the same graph manifold block but intersect the same block $\widetilde{M}^{\hb}_k$, then by Strong Separation (a) of Theorem~\ref{thm:cubulation}(5) they are at bounded distance hence there are finitely many $\mathrm{Stab}(\widetilde{S})$ orbits as well. This proves Condition~(3) of Criterion~\ref{thm:HWcriterion}. Hence Hypothesis~(i) of Theorem~\ref{thm:specialization} is satisfied.

Hypothesis~(ii) of Theorem~\ref{thm:specialization} coincides with \cite[Cor 4.2]{PW} which is a particular case of \cite[Thm 1]{Ham}. To verify Hypothesis~(iii) we need the following.

\begin{lemma}
\label{lem:far->disjoint}
Let $\mathcal{S}^{\hb}$ be a finite family of geometrically finite immersed incompressible surfaces in a compact hyperbolic $3$--manifold $M^{\hb}$. There exists  $R$ such that if the stabilizer of an elevation $\widetilde{S}$ to $\widetilde{M}^{\hb}$ of a surface in $\mathcal{S}^{\hb}$ intersects a stabilizer of a boundary plane $\widetilde{T}\subset \partial \widetilde M^{\hb}$ along an infinite cyclic group, then $N=N_R(\widetilde{S})\cap \widetilde{T}$ is nonempty.

Moreover, assume that we have two such elevations $\widetilde{S},\widetilde{S}'$ of possibly distinct surfaces.
If $\widetilde{S}\cap \widetilde{T}$ and $\widetilde{S}'\cap \widetilde{T}$ are nonempty and at distance $\geq R$ in the intrinsic metric on $\widetilde{T}$
(resp.\ $N$ and $N'=N_R(\widetilde{S}')\cap \widetilde{T}$ are sufficiently far with respect to some specified $r$),
then $\widetilde{S}$ and $\widetilde{S}'$ are disjoint (resp.\ at distance $\geq r$) and $\mathcal T(\widetilde{S})\cap \mathcal T(\widetilde{S}')\subset\{\widetilde{T}\}$.
\end{lemma}

Note that $N\subset \widetilde{T}$ is at a finite Hausdorff distance from a line, since the intersection of the stabilizers of $\widetilde{S}$ and $\widetilde{T}$ is infinite cyclic.

\begin{proof}
We can assume that the Riemannian metric on $M^{\hb}$ is hyperbolic and the toroidal boundary components are horospherical.
The first assertion follows from the fact that the surfaces in $\mathcal{S}^{\hb}$ have finitely many maximal parabolic subgroups.
For the second assertion, note that all elevations $\widetilde{S}\subset \widetilde{M}^{\hb}$ of surfaces in $\mathcal{S}^{\hb}$ are relatively quasiconvex.
It is well known (see e.g.\ \cite[Prop 3.1]{MP}), that each $\widetilde{S}^\star=\widetilde{S}\cup (\mathcal T(\widetilde{S})-\widetilde{T})$ is relatively quasiconvex as well. Then (see e.g.\ \cite[Lem 3.11]{Wen}) there is a constant $R$ such that the nearest point projection $\Pi$ onto $\widetilde{T}$ maps $\widetilde{S}^\star$ into $N_R(\widetilde{S}^\star)\cap \widetilde{T}$. Since $N_R(\widetilde T')\cap \widetilde T$ is bounded for each $\widetilde T'\in \mathcal T(\widetilde{S})-\widetilde{T}$, after increasing $R$ we can assume $\Pi(\widetilde{S}^\star)\subset N=N_R(\widetilde{S})\cap\widetilde{T}$. Since $\Pi$ is $1$--Lipschitz, the distance between $\widetilde{S}$ and $\widetilde{S}'$ is bounded below by the distance between $N$ and $N'$. This proves the second assertion.
\end{proof}

We now verify Hypothesis~(iii), by appealing to Criterion~\ref{thm:HWcriterion}. The action of $P_i$ on $\widetilde{Y}_i$ is free.
By the choice of $\mathcal{U}_i$ in the definition of $\widetilde{Y}_i$, any hyperplane $\widetilde{A}$ intersecting $\widetilde{Y}_i$ corresponds to a surface $\widetilde{S}\in \widetilde{\mathcal{S}}$ that for some $R>0$ has $N_R(\widetilde{S})\cap \widetilde{M}^{\gb}_i$ of infinite diameter. Consequently $\mathrm{Stab}(\widetilde{S})$ nontrivially intersects $P_i=\mathrm{Stab}(\widetilde{M}^{\gb}_i)$. It suffices then to fix $R$ from Lemma~\ref{lem:far->disjoint}.

Condition~(1) of Criterion~\ref{thm:HWcriterion} is immediate. To prove Conditions (2)~and~(3), it suffices to justify the claim that any pair of surfaces $\widetilde{S},\widetilde{S}'\in \widetilde{\mathcal{S}}$ with $N=N_R(\widetilde{S})\cap \widetilde{M}^{\gb}_i,N'=N_R(\widetilde{S}')\cap \widetilde{M}^{\gb}_i$ sufficiently far is separated by another surface in $\widetilde{\mathcal{S}}$. If both $\widetilde{S},\widetilde{S}'$ intersect $\widetilde{M}^{\gb}_i$, then the claim follows from Strong Separation~(b) in Theorem~\ref{thm:cubulation}(5). Otherwise, if one of $\widetilde{S},\widetilde{S}'$ is disjoint from $\widetilde{M}^{\gb}_i$ and they are not separated by a JSJ plane, then they both intersect a hyperbolic block $\widetilde{M}^{\hb}_k$ adjacent to $\widetilde{M}^{\gb}_i$. By Lemma~\ref{lem:far->disjoint}, for every $r$, if $N$ and $N'$ are sufficiently far, then $\widetilde{S}\cap \widetilde{M}^{\hb}_k$ and $\widetilde{S}'\cap\widetilde{M}^{\hb}_k$ are at distance $\geq r$ and $\mathcal{T}(\widetilde{S}\cap \widetilde{M}^{\hb}_k)\cap \mathcal{T}(\widetilde{S}'\cap \widetilde{M}^{\hb}_k)=\emptyset$. Then the claim follows from Strong Separation~(a). As a consequence of Condition~(2) the complex $\widetilde{Y}_i$ has finitely many $P_i$ orbits of codim--$2$--hyperplanes.

The nontrivial stabilizers in $P_i$ of hyperplanes in $\widetilde{Y}_i$ correspond to either fundamental groups of the pieces of $S$ in $M^{\gb}_i$, which are virtually embedded in $M^{\gb}_i$ by Theorem~\ref{thm:cubulation}(3) or infinite cyclic subgroups of the fundamental groups of the transitional tori, to which by \cite[Cor~4.3]{PW} (or \cite{Ham}) we can also associate virtually embedded $\partial$--parallel annuli. All these stabilizers are separable by \cite[Thm 1.1]{PW} and double coset separable by \cite[Thm 1.2]{PW}. Hence we have Conditions (4)~and~(5) of Criterion~\ref{thm:HWcriterion}, and by Criterion~\ref{thm:HWcriterion} the action of $P_i$ on $\widetilde{Y}_i$ is virtually special. This is Hypothesis~(iii).

Hypothesis~(iv) follows from the following, where we do not assume that $\partial M$ is toroidal. While this strengthening is not spelled out in \cite{Hier}, the proof goes through without a change.

\begin{thm}[{\cite[Thm 14.29]{Hier}}]
\label{thm:hyperbolic->special}
Let $M$ be a compact hyperbolic manifold with nonempty boundary. Then $\pi_1M$ is virtually compact special.
\end{thm}

\section{Surfaces in graph manifold blocks}
\label{sec:graph}
The goal of the next three sections is to prove Theorem~\ref{thm:cubulation} (Cubulation). We first review the existence results for surfaces in graph manifolds with boundary. Let $M^{\gb}$ be a \emph{graph manifold}, i.e.\ a compact connected oriented irreducible $3$--manifold with only Seifert fibered blocks in its JSJ decomposition. Assume $\partial M^{\gb}\neq\emptyset$. If $M^{\gb}$ is Seifert fibered, then an immersed incompressible surface $S\rightarrow M^{\gb}$ is \emph{horizontal} if it is transverse to the fibers and \emph{vertical} if it is a union of fibers. An immersed incompressible surface $S\rightarrow M^{\gb}$ that is not a $\partial$--parallel annulus is assumed to be homotoped so that its pieces are horizontal or vertical.

\begin{prop}[{\cite[Prop 3.1]{PW}}]
\label{prop:existence_graph}
Let $M^{\gb}$ be a graph manifold with $\partial M^{\gb}\neq\emptyset$. There exists a finite cover $\widehat{M}^{\gb}$ with a finite family $\mathcal{S}^{\gb}$ of embedded incompressible surfaces that are not $\partial$--parallel annuli such that:
\begin{itemize}
\item for each block $\widehat{B}\subset \widehat{M}^{\gb}$ and each torus $T\subset \partial \widehat{B}$, there is a surface $S\in \mathcal{S}^{\gb}$ such that $S\cap T$ is nonempty and vertical with respect to $\widehat{B}$,
\item for each block $\widehat{B}\subset \widehat{M}^{\gb}$ there is a surface $S\in \mathcal{S}^{\gb}$ such that $S\cap \widehat{B}$ is horizontal.
\end{itemize}
Every block $\widehat{B}\subset\widehat{M}^{\gb}$ is a product of a circle and a surface.
\end{prop}

Let $\mathcal{F}$ be a family of properly embedded essential arcs and curves in a compact hyperbolic surface $\Sigma$ with geodesic boundary. We say that $\mathcal{F}$ \emph{strongly fills} (resp.\ \emph{fills}) $\Sigma$ if the complementary components on $\Sigma$ of the geodesic representatives of the arcs and curves in $\mathcal{F}$ are discs (resp.\ discs or annuli parallel to the components of $\partial \Sigma$). This does not depend on the choice of the hyperbolic metric on $\Sigma$.

\begin{constr}
\label{constr:existence_graph}
Let $M^{\gb}$ be a non-thin graph manifold with $\partial M^{\gb}\neq\emptyset$. Consider $\widehat{M}^{\gb}$ and $\mathcal{S}^{\gb}$ satisfying Proposition~\ref{prop:existence_graph}. Add the following surfaces to $\mathcal{S}^{\gb}$:
\begin{itemize}
\item
all JSJ tori of $\widehat{M}^{\gb}$,
\item
vertical tori in each block $\widehat{B}\subset \widehat{M}^{\gb}$, whose base curves fill $\Sigma$.
\end{itemize}
Then the base arcs and curves of the vertical pieces of $\mathcal{S}^{\gb}$ in $\widehat{B}$ strongly fill $\Sigma$.
We retain the notation $\mathcal{S}^{\gb}$ for the projection of this extended family to $M^{\gb}$.
\end{constr}

\begin{rem}
\label{rem:existence_graph}
Let $\mathcal{S}^{\gb}$ be the family from Construction~\ref{constr:existence_graph}. Let $B$ be a block of $M^{\gb}$. For each $g\in\pi_1B$, some piece of $\mathcal{S}^{\gb}$ in $B$ is a cut-surface for $g$.
\end{rem}

To prove Strong Separation in Theorem~\ref{thm:cubulation} (Cubulation) we will need the following WallNbd-WallNbd Separation property in blocks.

\begin{defin}[{\cite[\S 8.3]{HruW}}]
Let $\mathcal{S}$ be a family of immersed incompressible surfaces in a compact Riemannian $3$--manifold $M$. Let $\widetilde{\mathcal{S}}$ be the family of elevations of the surfaces in $\mathcal{S}$ to the universal cover $\widetilde{M}$ of $M$. The family $\mathcal{S}$ has \emph{WallNbd-WallNbd Separation} if for any $r$ there is $d=d(r)$ such that if $\widetilde{S},\widetilde{S}'\in\widetilde{\mathcal{S}}$ have neighborhoods $N_r(\widetilde{S}), N_r(\widetilde{S}')$ at distance $\geq d$, then $N_r(\widetilde{S}), N_r(\widetilde{S}')$ are separated by a surface in $\widetilde{\mathcal{S}}$. This property is independent of the choice of Riemannian metric, but the value of $d$ might vary.
Similarly $\mathcal{S}$ has \emph{Ball-Ball Separation} if for any $r$ there is $d$ such that each pair of metric $r$--balls at distance $\geq d$ is separated by a surface in $\widetilde{\mathcal{S}}$.

We analogously define WallNbd-WallNbd Separation and Ball-Ball Separation for a family of essential arcs and curves in a compact hyperbolic surface.
\end{defin}

The following is easy to prove directly, but for uniformity of our arguments, we will deduce it in Section~\ref{sec:hyperbolic} from Criterion~\ref{thm:WallNbd}.

\begin{lem}
\label{lem:wall_separation}
A strongly filling family of arcs and curves in a hyperbolic surface satisfies WallNbd-WallNbd Separation.
Consequently, if their base arcs and curves strongly fill, then the vertical pieces of $\mathcal{S}^{\gb}$ in $B$ satisfy WallNbd-WallNbd Separation.
\end{lem}

Let $\Sigma$ be the base orbifold of a non-thin block $B\subset M^{\gb}$. The fundamental groups of the components of $\partial \Sigma$ intersect trivially. Thus by the compactness of the base arcs of the annular vertical pieces of $\mathcal{S}^{\gb}$ in $B$ we have the following analogue of Lemma~\ref{lem:far->disjoint}.

\begin{rem}
\label{rem:far->disjoint}
Let $\mathcal{S}^{\gb}$ be a finite family of immersed incompressible surfaces in a non-thin graph manifold $M^{\gb}$. There exists $R$ with the following property. Let $B\subset M^{\gb}$ be a block with elevation $\widetilde{B}\subset \widetilde{M}^{\gb}$ and let $\widetilde{S},\widetilde{S}'$ be elevations to $\widetilde{M}^{\gb}$ of surfaces in $\mathcal{S}^{\gb}$. Suppose that $\widetilde{S}_o=\widetilde{S}\cap \widetilde{B}$ and $\widetilde{S}'_o=\widetilde{S}'\cap\widetilde{B}$ are both vertical, and suppose that there is a plane $\widetilde{T}\subset\partial \widetilde{B}$
intersecting both $\widetilde{S}_o$ and $\widetilde{S}'_o$. If the distance between the lines $\widetilde{S}_o\cap\widetilde{T}$ and $\widetilde{S}'_o\cap \widetilde{T}$ is $\geq R$ in the intrinsic metric on $\widetilde{T}$, then
$\widetilde{S}_o$ and $\widetilde{S}_o'$ are disjoint and $\mathcal T(\widetilde{S}_o)\cap \mathcal T(\widetilde{S}_o')=\{\widetilde{T}\}$.
\end{rem}

In \cite{PW} we also established the following.
\begin{cor}[{\cite[Cor 3.3]{PW}}]
\label{cor:existence_graph}
Let $M^{\gb}$ be a graph manifold with $\partial M^{\gb}\neq\emptyset$. There exists a finite cover $\widehat{M}^{\gb}$ of $M^{\gb}$ such that for each essential circle $C$ in a torus $T\subset \partial \widehat{M}^{\gb}$ there is an incompressible surface $S_C$ embedded in $\widehat{M}^{\gb}$ with $S_C\cap T$ consisting of a nonempty set of circles parallel to $C$.
\end{cor}

Finally, the following holds by \cite[Thm 2.3]{RW}.
\begin{lem}
\label{lem:covers of emb are emb}
Let $S$ be an incompressible surface embedded in a graph manifold $M^{\gb}$. Let $S'\rightarrow S$ be a finite cover. Then $S'\rightarrow M^{\gb}$ is virtually embedded.
\end{lem}

\section{Surfaces in hyperbolic blocks}
\label{sec:hyperbolic}

We now review the existence results for surfaces in hyperbolic blocks.
First we establish a hyperbolic analogue of Proposition~\ref{prop:existence_graph}.

\begin{thm}[{compare \cite[Cor 14.33]{Hier}}]
\label{thm:hyperbolic->cubul specially by surfaces}
Let $M^{\hb}$ be a compact hyperbolic $3$--manifold with nonempty boundary. There is in $M^{\hb}$ a finite family $\mathcal{S}^{\hb}$ of geometrically finite immersed incompressible surfaces containing cut-surfaces for all elements of $\pi_1M^{\hb}$. Moreover, the surfaces have no accidental parabolics, i.e.\ any parabolic element in $\pi_1S$ with $S\in\mathcal{S}^{\hb}$ lies in $\pi_1C$ for some component $C$ of $\partial S$.
\end{thm}
\begin{proof}
We follow the proof of \cite[Lem 14.32]{Hier}.
By Theorem~\ref{thm:hyperbolic->special}, without loss of generality we can assume $\pi_1M^{\hb}=\pi_1X$ for a compact special cube complex $X$. Since $\pi_1M^{\hb}$ acts freely on the universal cover $\widetilde{X}$ of $X$, for every $g\in \pi_1M^{\hb}$ there is a CAT(0) geodesic axis $\R\subset \widetilde{X}$. Let $\widetilde{B}\subset \widetilde{X}$ be any hyperplane intersected transversely by $\gamma$.
The subgroup $\mathrm{Stab}(\widetilde{B})\subset \pi_1M^{\hb}$ is geometrically finite, since otherwise by Tameness \cite{Atam,CG} and Covering \cite{Thu,Can} Theorems it would be virtually a fiber subgroup and would not admit a virtual retraction guaranteed by Theorem~\ref{thm:canonical_completion_and_retraction}.

Let $\widetilde{M}^{\hb}\subset \H^3$ be the universal cover of $M^{\hb}$. The boundary of the hyperbolic convex core $N$ of the $\mathrm{Stab}(\widetilde{B})$ cover of $M^{\hb}$ consists of finitely many geometrically finite surfaces. Suppose first that $g$ is hyperbolic. Since $\widetilde{M}^{\hb}$ and $\widetilde{X}$ are quasi-isometric, the geodesic axis $\R$ for $g$ in $\widetilde{M}^{\hb}$ intersects an elevation $\widetilde{N}$ of $N$. Thus there is an elevation $\widetilde{S}\subset \partial \widetilde N$ of a surface $S\subset \partial N$ intersecting $\R$ as well. Hence $S$ is a cut-surface for $g$. If $S$ contains essential circles $C_j$ not homotopic into $\partial S$ with $\pi_1C_j$ parabolic, then let $A_j$ be annuli joining $C_j$ to the boundary. Since all $A_j$ lie on the same side of $S$ as $N$, the circles $C_j$ are disjoint up to homotopy. Then a boundary surface of the convex core for one of the components of $S-\bigcup_jC_j$ is a cut-surface for $g$ and has no accidental parabolics. In the case where $g$ is parabolic, it suffices to use an axis $\R$ for $g$ on a horosphere.
\end{proof}

\subsection{WallNbd-WallNbd Separation}
We now describe a tool from \cite{HruW} for verifying WallNbd-WallNbd Separation in relatively hyperbolic spaces.

\begin{defin}
\label{def:WallNbd_relative}
Let $\mathcal{S}$ be a finite family of immersed incompressible surfaces in a compact Riemannian $3$--manifold $M$. Let $\widetilde{\mathcal{S}}$ be the family of elevations of the surfaces in $\mathcal{S}$ to the universal cover $\widetilde{M}$ of $M$. Let $T\subset M$  be a connected subspace, and let $\widetilde{T}$ be its elevation to $\widetilde{M}$.

We say that $\mathcal{S}$ satisfies \emph{WallNbd-WallNbd Separation} in $T$ if for any $r$ there is $d$ such that if $\widetilde{S},\widetilde{S}'\in\widetilde{\mathcal{S}}$ have nonempty $N=N_r(\widetilde{S})\cap \widetilde{T}, N'=N_r(\widetilde{S}')\cap \widetilde{T}$ at distance $\geq d$, then $N$ and $N'$ are separated in $\widetilde{T}$ by a restriction to $\widetilde{T}$ of a surface in $\widetilde{\mathcal{S}}$.

We say that $\mathcal{S}$ satisfies \emph{Ball-WallNbd Separation} in $T$ if for any $r$ there is $d$ such that if $\widetilde{S}\in\widetilde{\mathcal{S}}$ has nonempty $N=N_r(\widetilde{S})\cap \widetilde{T}$ and $m$ is a point of $\widetilde{T}$ with $N'=N_r(m)\cap \widetilde{T}$ at distance $\geq d$ from  $N$, then $N$ and $N'$ are separated in $\widetilde{T}$ by a restriction to $\widetilde{T}$ of a surface in $\widetilde{\mathcal{S}}$.
\end{defin}

\begin{cri}[{\cite[Cor 8.10]{HruW}}]
\label{thm:WallNbd}
Let $\mathcal{S}$ be a finite family of:
\begin{enumerate}[(a)]
\item essential arcs and curves in a compact hyperbolic surface $M$ satisfying Ball-Ball Separation, or

\item immersed incompressible surfaces in a compact Riemannian $3$--manifold $M$. Let $T_i\subset M$ be connected subspaces. Suppose that $\pi_1M$ is hyperbolic relative to $E_i\subset \pi_1M$ that are the images of $\pi_1T_i$. Assume that $\pi_1S$ is relatively quasiconvex for each $S\in \mathcal{S}$. Suppose that $\mathcal{S}$ satisfies Ball-Ball Separation in $M$ and WallNbd-WallNbd Separation and Ball-WallNbd Separation in all $T_i$.
\end{enumerate}
Then $\mathcal{S}$ satisfies WallNbd-WallNbd Separation in $M$.
\end{cri}

The hypothesis of Ball-Ball Separation can be verified using the following.

\begin{lemma}[{\cite[Lem 5.3]{HruW}}]
\label{lem:ball separation}
Let $\mathcal{S}$ be a finite family of:
\begin{enumerate}[(a)]
\item
essential arcs and curves in a compact hyperbolic surface $M$, or
\item
immersed incompressible surfaces in a compact Riemannian $3$--manifold $M$.
\end{enumerate}
If the action of $\pi_1M$ on the associated dual CAT(0) cube complex is free, then $\mathcal{S}$ satisfies Ball-Ball Separation.
\end{lemma}

Consequently, Criterion~\ref{thm:WallNbd}(a) and Lemma~\ref{lem:ball separation}(a) yield Lemma~\ref{lem:wall_separation}.

\begin{cor}
\label{cor:WallNbd_hyperbolic}
Let $\mathcal{S}^{\hb}$ be a finite family of geometrically finite surfaces in a hyperbolic $3$--manifold $M^{\hb}$. Suppose that:
\begin{enumerate}[(i)]
\item
for each $g\in \pi_1M^{\hb}$ there is a cut-surface for $g$ in $\mathcal{S}^{\hb}$, and
\item
for each parabolic element $g\in \pi_1S$ with $S\in \mathcal{S}^{\hb}$ there is a surface $S'\in \mathcal{S}^{\hb}$ with a curve $C\subset\partial S'$ such that $g^n$ is conjugate to an element of $\pi_1C$ for some $n\neq 0$.
\end{enumerate}
Then $\mathcal{S}$ satisfies WallNbd-WallNbd Separation in $M^{\hb}$.
\end{cor}

Let $\partial_t M^{\hb}\subset \partial M^{\hb}$ denote the union of toroidal boundary components.

\begin{proof}
We verify the hypotheses of Criterion~\ref{thm:WallNbd}(b). Ball-Ball Separation in $M^{\hb}$ follows from Hypothesis~(i) and Lemma~\ref{lem:ball separation}(b). We now verify WallNbd-WallNbd Separation and Ball-WallNbd Separation in a torus $T_i$, where $\bigsqcup_i T_i=\partial_tM^{\hb}$. Consider an elevation $\widetilde{T}$ of $T_i$ to the universal cover $\widetilde{M}^{\hb}$ of $M^{\hb}$. For any $r$ there is $d$ such that the intersections $N_r(\widetilde{S})\cap \widetilde{T}$ or $N_r(m)\cap \widetilde{T}$ from the definition of WallNbd-WallNbd Separation and Ball-WallNbd Separation are either of diameter $\leq d$ or at Hausdorff distance $\leq d$ from a line $\widetilde{C}$ in $\widetilde{T}$ stabilized by $g\in \pi_1S\cap\pi_1T_i$. For $S'\in \mathcal{S}^{\hb}$, the intersections $\widetilde{S'}\cap \widetilde{T}$ are infinite families of parallel lines. By Hypothesis~(ii), their directions include the direction of each $\widetilde{C}$ above. This yields WallNbd-WallNbd Separation and Ball-WallNbd Separation in $T_i$.
\end{proof}

\subsection{Capping off surfaces}
We will need one more crucial piece of information concerning the existence of surfaces in hyperbolic blocks with designated boundary circles.

\begin{prop}
\label{prop:capping off}
Let $M^{\hb}$ be a compact hyperbolic $3$--manifold and let $C_0,\ldots,C_n$ be essential circles in the tori $T_0\sqcup\cdots\sqcup T_n=\partial_t M^{\hb}$.
There exists a geometrically finite immersed incompressible surface $S\rightarrow M^{\hb}$ with $S\cap\partial_tM^{\hb}$ covering $C_0$ and such that all the parabolic elements of $\pi_1S$ are conjugate to $\pi_1C_i$.
\end{prop}

In the proof we will need the following relative version of the Special Quotient Theorem.

\begin{thm}[{\cite[Lem 16.13 and Rem 6.14]{Hier}}]
\label{thm:relative special qoutient}
Let $G$ be a compact special group that is hyperbolic relative to free-abelian subgroups $\{E_i\}$. Then there are finite index subgroups $E^\circ_i\subset E_i$ such that for any further subgroups $E^c_i\subset E^\circ_i$ with $E_i/E^c_i$ finite or virtually cyclic, the quotient $G/\nclose{\{E^c_i\}}$ is hyperbolic and virtually compact special. Moreover, each $E_i/E^c_i$ embeds into $G/\nclose{\{E^c_i\}}$.
\end{thm}

We will also use the following combination theorem.

\begin{thm}[{\cite[Thm 1.1]{MP}}]
\label{thm:Eduardo}
Let $S\subset M^{\hb}$ be an incompressible geometrically finite surface in a hyperbolic manifold $M^{\hb}$.
Let $C_j$ be components of $\partial S$ contained in boundary tori $T_{i_j}$ of $M^{\hb}$ (some $T_{i_j}$ may coincide).
Then for almost all cyclic covers $T'_j$ of $T_{i_j}$ to which $C_j$ lift, the fundamental group $\pi_1S^\star$ of the graph of spaces $S^\star$ obtained by amalgamating $S$ with $T'_j$ along $C_j$
embeds in $\pi_1M^{\hb}$ and is relatively quasiconvex. Moreover, every parabolic subgroup of $\pi_1S^\star$ is conjugate in $\pi_1S^\star$ to a subgroup of $\pi_1S$ or $\pi_1T'_j$.
\end{thm}

\begin{proof}[Proof of Proposition~\ref{prop:capping off}]
By Theorem~\ref{thm:hyperbolic->special} without loss of generality we can assume that $G=\pi_1M^{\hb}$ is compact special. By Theorem~\ref{thm:relative special qoutient}, there are $g_i\in \pi_1C_i$ such that
$\overline{G}=G/\nclose{\{g_i\}}$ is hyperbolic and virtually compact special. For a subgroup $F\subset G$ we denote by $\overline{F}$ its image in $\overline{G}$. By Theorem~\ref{thm:relative special qoutient} we obtain additionally that for each $E=\pi_1 T$ with $T\subset \partial_tM^{\hb}$ the quotient $\overline{E}\subset \overline{G}$ is infinite. We will prove that there is a finite index normal subgroup $G'\subset G$ satisfying the following.
\begin{enumerate}[(i)]
\item For each $E=\pi_1 T$ the image of $\overline{E\cap G'}\rightarrow \homology_1(\overline {G'})$ has rank $1$.

\item The cover $M'$ of $M^{\hb}$ corresponding to $G'$ has $$\mathrm{rk}\ \homology_2(M')/\mathrm{im}\big(\homology_2(\partial M')\rightarrow \homology_2(M')\big)\geq 2.$$
\end{enumerate}

Properties~(i) and~(ii) are preserved when passing to further finite covers, so it suffices to achieve them separately. To obtain Property~(i), let $E=\pi_1 T$. By canonical completion and retraction (see Theorem~\ref{thm:canonical_completion_and_retraction}), there is a finite index subgroup $\overline{G}'\subset \overline{G}$ that retracts onto an infinite cyclic subgroup $\Z\subset \overline{E}\cap \overline{G}'$. Thus $\Z$ embeds in $\homology_1(\overline{G}')$. The preimage $G'\subset G$ of $\overline{G}'\subset \overline{G}$ satisfies Property~(i) for the specified $E$.
Property~(ii) follows directly from \cite[Cor 1.4]{CLR}.

By Property~(i), there is a map $f_*\colon G'\rightarrow \Z$ factoring through $\homology_1(\overline{G}')$ and with $f_*$
nontrivial on each $g^{-1}Eg\cap G'$ with $g\in G$. Let $f\colon M'\rightarrow S^1$ be a map inducing $f_*$. By Sard's theorem, there is a point $s\in S^1$ so that $S'=f^{-1}(s)\subset M'$ is a properly embedded surface, possibly disconnected. Then $S'\cap \partial_tM'$ is a union of families of identically oriented circles covering $C_i$. We compress $S'$ to an incompressible surface with the same boundary.

We now claim that in the case where $S'$ is a fiber, without changing $\partial S'$ we can change $S'$ so that it is geometrically finite. In $\homology_2(M',\partial M';\R)$ we consider the Thurston norm ball $B_x$, see \cite{Thnorm}. Let $L\subset \homology_2(M',\partial M';\R)$ be the subspace of homology classes whose restriction to $\homology_1(\partial M';\R)$ is proportional to $[\partial S']$. By Property~(ii) the rank of the image of $\homology_2(M')$ in $\homology_2(M',\partial M')$ is $\geq 2$. Hence the dimension of $L$ is $\geq 2$. We can take $S'$ represented by a point on a ray not passing through a maximal face of $B_x\cap L$. By \cite[Thm 3]{Thnorm}, the surface $S'$ is not a fiber of $M'$. By Covering \cite{Thu} and Tameness \cite{BonT}, the surface $S'$ is geometrically finite, proving the claim.

Let $S'_o$ be a component of $S'$ intersecting $T_0$. Note that each parabolic element of $\pi_1S'_o$ is conjugate into $\pi_1C'$ for some component $C'\subset\partial S'$ since $S'$ intersects all the boundary tori of $M'$.
Consequently all parabolic elements of $\pi_1S'_o$ are conjugate into $\pi_1C_i$.

Since $S'_o$ is geometrically finite, Theorem~\ref{thm:Eduardo} applies to $S'_o$. Let $\{C'_j\}$ be the components of $\partial S'_o$ outside $T_0$, and for each $j$ let $T_{i_j}$ denote the torus containing $C'_j$. Let $S^\star$ be the graph of spaces obtained by amalgamating $S'_o$ along $C'_j$ with $T'_j$ provided by Theorem~\ref{thm:Eduardo}.
Then $S^\star\rightarrow M'$ extends to an immersion $N\rightarrow M'$, where $N$ is the regular neighborhood of $S^\star$ in the $\pi_1S^\star$ cover of $M'$. Let $S$ be the non-toroidal component of $\partial N$, that is the component that can be decomposed into two surfaces parallel to $S'_o$ that are combined along curves parallel to $C'_j$. The immersed surface $S\rightarrow M'$ is incompressible, since $\pi_1S$ embeds in $\pi_1S^\star$.
Moreover $S$ is not a virtual fiber since $\pi_1 S^\ast$ is relatively quasiconvex and of infinite index in $\pi_1M^{\hb}$.
By Theorem~\ref{thm:Eduardo} every parabolic element of $\pi_1S$ is conjugate in $\pi_1S^\star$ to a parabolic element of $\pi_1S'_o$ or into one of the $\pi_1T_j'$. But each parabolic element of $\pi_1S'_o$ is conjugate into some $\pi_1C_i$. Moreover, the intersection of $\pi_1S$ with a $\pi_1S^\star$ conjugate of $\pi_1T'_j$ lies in a conjugate of $\pi_1C'_j$. Thus the immersion $S\rightarrow M'\rightarrow M^{\hb}$ has the desired property for parabolic elements.
\end{proof}

\section{Cubulation}
\label{sec:cubul}

In this section we combine the surfaces described in the graph manifold blocks and hyperbolic blocks. To prove Theorem~\ref{thm:cubulation} (Cubulation) we need the following:

\begin{lem}
\label{lem:degrees}
Let $S$ be a connected compact surface with $\chi(S)<0$.
There exists $K=K(S)$ such that for each assignment of a positive integer $n_C$ to each boundary circle $C\subset\partial S$, there is a connected finite
cover $\widehat{S}\rightarrow S$ whose degree on each component of the preimage of $C$ equals $Kn_C$.
\end{lem}

We can allow $S$ to be disconnected. We can also allow annular components, but obviously require that the integers $n_C$ coincide for both boundary circles of such a component.

\begin{proof}
Let $K=K(S)$ be the degree of a cover of $S$ with nonzero genus. The lemma follows from \cite[Lem 4.7]{PW}.
\end{proof}

\begin{proof}[Proof of Theorem~\ref{thm:cubulation}]
The proof has two steps. In the first step we construct a family $\mathcal{S}$ of surfaces satisfying Theorem~\ref{thm:cubulation}(1)--(4). In the second step we prove that $\mathcal{S}$ satisfies the Strong Separation property in Theorem~\ref{thm:cubulation}(5).

\smallskip

\noindent \textbf{Construction.}
Let $\mathcal{S}^{\hb}_k$ be the family of surfaces in $M^{\hb}_k$ given by Theorem~\ref{thm:hyperbolic->cubul specially by surfaces}. Let $\mathcal{C}$ be the family of circles embedded in the transitional tori of $M$
that are covered by the boundary circles of the surfaces in $\{\mathcal{S}^{\hb}_k\}$ up to homotopy on the tori. Note that every transitional torus contains circles of $\mathcal{C}$. Let $\mathcal{C}_i\subset \mathcal{C}$ be the circles lying in $\partial M^{\gb}_i$.

By Corollary~\ref{cor:existence_graph} each circle $C\in\mathcal{C}_i$ is covered by a boundary circle of an immersed incompressible surface $S^{\gb}_C\rightarrow M^{\gb}_i$ virtually embedded in $M^{\gb}_i$.
Let $\mathcal{S}^{\gb}_i$ be the family of surfaces in $M^{\gb}_i$ provided by Construction~\ref{constr:existence_graph} and
let $\mathcal{S}'^{\gb}_i=\mathcal{S}^{\gb}_i\cup\{S^{\gb}_C\}_{C\in \mathcal{C}_i}$.

Let $\mathcal{C}'$ denote the family of circles embedded in the transitional tori of $M$ covered (up to homotopy) by the boundary circles of the surfaces in $\{\mathcal{S}'^{\gb}_i\}$.
Let $\mathcal{C}'_k\subset \mathcal{C}'$ be the circles lying in $\partial M^{\hb}_k$.
By Proposition~\ref{prop:capping off},
for each circle $C'\in\mathcal{C}'_k$ there is a geometrically finite immersed incompressible surface $S^{\hb}_{C'}\rightarrow M^{\hb}_k$ such that $S^{\hb}_{C'}\cap\partial_tM^{\hb}_k$ is nonempty and covers $C'$. Moreover, we require that all the parabolic elements of $\pi_1S^{\hb}_{C'}$ are conjugate to $\pi_1C$ with $C\in \mathcal{C}'$. Let
$\mathcal{S}'^{\hb}_k=\mathcal{S}^{\hb}_k\cup \{S^{\hb}_{C'}\}_{C'\in \mathcal{C}'_k}$.

We will apply Lemma~\ref{lem:degrees} to produce families of surfaces $\{\widehat{\mathcal{S}}'^{\hb}_k\}, \{\widehat{\mathcal{S}}'^{\gb}_i\}$ covering $\{\mathcal{S}'^{\hb}_k\},\{\mathcal{S}'^{\gb}_i\}$ such that $\mathcal{S}_o=\{\widehat{\mathcal{S}}'^{\hb}_k\}\cup\{\widehat{\mathcal{S}}'^{\gb}_i\}$ has the following property: There is a uniform $d$ such that for each circle in $\mathcal{C'}$, each component of its preimage in a surface in $\mathcal{S}_o$ covers it with degree $d$.
In order to arrange this, for a boundary circle $C$ of a surface in
$\{\mathcal{S}'^{\hb}_k\}\cup\{\mathcal{S}'^{\gb}_i\}$ let $d_C$ denote the degree with which $C$ maps onto a circle in $\mathcal{C}'$. Let $n_C=\frac{1}{d_C}\prod_Cd_C$. Applying Lemma~\ref{lem:degrees} with this choice of $\{n_C\}$ provides the uniform $d=K\prod_Cd_C$. Note that for an annular surface the degrees $d_C$ coincide and hence the numbers $n_C$ coincide. We can then take a cyclic cover.

We will now extend each surface $S_o\in\mathcal{S}_o$ to a surface immersed properly in $M$ by combining appropriately many copies of other surfaces in
$\mathcal{S}_o$. First assume $S_o\in \widehat{\mathcal{S}}'^{\gb}_i$. Let $\mathcal{C}'_o\subset \mathcal{C}'$ denote the set of circles covered by the boundary components of $S_o$ and let $m_{C'}$ denote the number of components of $S_o$ mapping to the circle $C'\in \mathcal{C}'_o$. Denote by $\widehat{S}^{\hb}_{C'}$ the surface in $\mathcal{S}_o$ covering $S^{\hb}_{C'}$ and by $l_{C'}$ the number of boundary components of $\widehat{S}^{\hb}_{C'}$ covering $C'$. Let $L=\prod_{C'\in \mathcal{C}'_o}l_{C'}$. Take $2L$ copies of $S_o$ and $2m_{C'}\frac{L}{l_{C'}}$ copies of $\widehat{S}^{\hb}_{C'}$, for each $C'$, with two opposite orientations. These surfaces combine to form a desired immersed incompressible surface extending $S_o$. Note that for each $C'\in \mathcal{C}'$ the surface $\widehat{S}^{\hb}_{C'}$ appears within such extension of some surface $S_o\in\widehat{\mathcal{S}}'^{\gb}_i$.

Hence it remains to consider the case $S_o\in \widehat{\mathcal{S}}^{\hb}_k\subset \widehat{\mathcal{S}}'^{\hb}_k$, where $\widehat{\mathcal{S}}^{\hb}_k$ is the family of surfaces covering the surfaces in $\mathcal{S}^{\hb}_k$. This case is treated similarly to the previous one. Let $\mathcal{C}_o\subset \mathcal{C}$ be the set of circles covered by the boundary components of $S_o$. Consider all the surfaces $\widehat{S}^{\gb}_C$ covering $S^{\gb}_C$ for $C\in \mathcal{C}_o$. Let $\mathcal{C}'_o\subset \mathcal{C}'$ denote the set of circles covered by the boundary components of these surfaces $\widehat{S}^{\gb}_C$. Consider all the surfaces $\widehat{S}^{\hb}_{C'}$, where $C'\in \mathcal C'_o$. Gluing the appropriate number of copies of $S_o, \widehat{S}^{\gb}_C$, and  $\widehat{S}^{\hb}_{C'}$ gives the desired extension.

We denote the union of both of these families of extended surfaces together with the family of the
JSJ tori by $\mathcal{S}$. So $\mathcal{S}$ obviously satisfies Theorem~\ref{thm:cubulation}(2). Observe that Theorem~\ref{thm:cubulation}(1) follows from Theorem~\ref{thm:cubulation}(2) and the existence of cut-surfaces in Theorem~\ref{thm:hyperbolic->cubul specially by surfaces} and Construction~\ref{constr:existence_graph}. The surfaces in $\widehat{\mathcal{S}}'^{\gb}_i$ are virtually embedded in $M^{\gb}_i$ by Lemma~\ref{lem:covers of emb are emb}, hence $\mathcal{S}$ satisfies Theorem~\ref{thm:cubulation}(3). The surfaces in $\widehat{\mathcal{S}}'^{\hb}_k$ are geometrically finite and thus $\mathcal{S}$ satisfies Theorem~\ref{thm:cubulation}(4).

We also record that by the way we have applied Proposition~\ref{prop:capping off} to construct $S^{\hb}_{C'}$, the pieces of $\mathcal{S}$ in every hyperbolic block $M_k^{\hb}$ satisfy Hypothesis~(ii) of Corollary~\ref{cor:WallNbd_hyperbolic}.

\vspace{2mm}
\noindent
\textbf{Strong Separation.}
We now verify Theorem~\ref{thm:cubulation}(5). We adopt the convention that in thin graph manifold blocks $T\times I$ we choose the vertical direction so that all the pieces in $T\times I$ of the surfaces in $\mathcal{S}$ are horizontal.
Let $R$ be a constant satisfying Lemma~\ref{lem:far->disjoint} and Remark~\ref{rem:far->disjoint} in all hyperbolic and non-thin Seifert fibered blocks of $M$, with respect to the pieces of $\mathcal{S}$.

We first prove Strong Separation (b). Suppose that $\widetilde{S},\widetilde{S}'\in\widetilde{\mathcal{S}}$ intersect a graph manifold block $\widetilde{M}^{\gb}_i\subset \widetilde{M}$. We need to show that if $\widetilde{S}\cap\widetilde{M}^{\gb}_i$ and $\widetilde{S}'\cap\widetilde{M}^{\gb}_i$ are sufficiently far, then they are separated by another surface in $\widetilde{\mathcal{S}}$.

First consider the case where $\widetilde{S}'$ intersects a JSJ or boundary plane $\widetilde{T}$ of $\widetilde{M}^{\gb}_i$ intersected by $\widetilde{S}$.
Since by Theorem~\ref{thm:cubulation}(1) the components of $S\cap M^{\gb}_i$ are virtually embedded, there is $h\in\mathrm{Stab}(\widetilde{T})$ such that the surfaces $\widetilde{S}\cap \widetilde{M}^{\gb}_i$ and $h\widetilde{S}\cap\widetilde{M}^{\gb}_i$ are disjoint. Moreover, by passing to a power of $h$ we can assume that they are at distance $\geq R$.
Let $\widetilde{M}_{\mathrm{hor}}\subset \widetilde{M}^{\gb}_i$ be the maximal graph manifold containing $\widetilde{T}$ such that $\widetilde{S}$ is horizontal in all the blocks of $\widetilde{M}_{\mathrm{hor}}$.
In the extreme cases $\widetilde M_{\mathrm{hor}}$ can equal $\widetilde M^{\gb}_i$ or $\widetilde T$.
Let $\widetilde{N}$ be the union of $\widetilde{M}_{\mathrm{hor}}$ with the adjacent hyperbolic and Seifert fibered blocks. By Lemma~\ref{lem:far->disjoint} and Remark~\ref{rem:far->disjoint} the surfaces $\widetilde{S}\cap \widetilde{N}$ and $h\widetilde{S}\cap \widetilde{N}$ are disjoint and the boundary lines of $\widetilde{S}\cap \widetilde{N}$ and $h\widetilde{S}\cap \widetilde{N}$ do not intersect a common JSJ plane outside $\widetilde{M}_{\mathrm{hor}}$. Hence the entire $\widetilde{S}$ and $h\widetilde{S}$ are disjoint.

For each $\widetilde{S}$ and $\widetilde{T}$ we fix $h$ as above. The surface $h\widetilde{S}\cap \widetilde{M}_{\mathrm{hor}}$ is in a bounded neighborhood of $\widetilde{S}\cap \widetilde{M}_{\mathrm{hor}}$. Hence if $\widetilde{S}'\cap \widetilde{M}^{\gb}_i$ is sufficiently far from $\widetilde{S}\cap \widetilde{M}^{\gb}_i$, then $\widetilde{S}'\cap \widetilde{M}_{\mathrm{hor}}$ is at distance $\geq R$ from each of $h^{\pm1}\widetilde{S}\cap\widetilde{M}_{\mathrm{hor}}$. As before $\widetilde{S}'$ is disjoint from both $h^{\pm1}\widetilde{S}$, and one of $h^{\pm1}\widetilde{S}$ separates $\widetilde{S}'$ from $\widetilde{S}$, as desired.
Since there are finitely many $\mathrm{Stab}(\widetilde{S}\cap\widetilde{M}^{\gb}_i)$ orbits of JSJ or boundary planes $\widetilde{T}$ of $\widetilde{M}^{\gb}_i$ intersected by $\widetilde{S}$, this argument works for all $\widetilde{T}$ simultaneously.

To complete the proof of Strong Separation (b) it remains to consider a second case where $\widetilde{S}'$ intersects a Seifert fibered block $\widetilde{M}_o\subset \widetilde{M}^{\gb}_i$ intersected by $\widetilde{S}$, but is disjoint from the JSJ and boundary planes of $\widetilde{M}^{\gb}_i$ intersected by $\widetilde{S}$. In that case the pieces of $\widetilde{S}$ and $\widetilde{S}'$ in $\widetilde M_o$ are vertical. Since the proof for Strong Separation (a) is the same, we perform it simultaneously: in that case $\widetilde{M}_o$ denotes the hyperbolic block $\widetilde{M}^{\hb}_k$. In both cases if we denote as usual by $\mathcal T(\widetilde{S}\cap \widetilde M_o)$ the set of JSJ and boundary planes in $\partial \widetilde M_o$ intersecting $\widetilde{S}$, then $\mathcal T(\widetilde{S}\cap \widetilde M_o)$ and $\mathcal T(\widetilde{S}'\cap \widetilde M_o)$ are disjoint.

As before, for any JSJ or boundary plane $\widetilde{T}\in \mathcal T(\widetilde{S}\cap \widetilde M_o)$ we fix $h\in \mathrm{Stab}(\widetilde{T})$ such that $\widetilde{S}$ and $h\widetilde{S}$ are disjoint. We do the same with $\widetilde{S}$ replaced by $\widetilde{S}'$. There is $R'$ such that for each $\widetilde{T}$, the translate $h\widetilde{S}\cap\widetilde{T}$ is contained in the $R'$--neighborhood of $\widetilde{S}\cap\widetilde{T}$ in the intrinsic metric on $\widetilde{T}$, and the same property holds with $\widetilde{S}$ replaced by $\widetilde{S}'$. Let $d=d(r)$ be a WallNbd-WallNbd Separation constant guaranteed by Lemma~\ref{lem:wall_separation} and Corollary~\ref{cor:WallNbd_hyperbolic} for $r=R+R'$ in all hyperbolic and non-thin Seifert fibered blocks with respect to the pieces of $\mathcal{S}$.

If the piece $\widetilde{S}'\cap\widetilde{M}_o$ is at distance $\geq 2r+d$ from the piece $\widetilde{S}\cap \widetilde{M}_o$, then by WallNbd-WallNbd Separation there is a surface $\widetilde{S}^*\in \widetilde{\mathcal{S}}$ such that $\widetilde{S}^*\cap \widetilde{M}_o$ separates $N_r(\widetilde{S}'\cap \widetilde{M}_o)$ from $N_r(\widetilde{S}\cap \widetilde{M}_o)$ in $\widetilde{M}_o$.
If $\mathcal T(\widetilde{S}^*\cap\widetilde{M}_o)$ is disjoint from $\mathcal T(\widetilde{S}'\cap \widetilde{M}_o)\cup \mathcal T(\widetilde{S}\cap\widetilde{M}_o)$, then $\widetilde{S}^*$ is disjoint from $\widetilde{S}'$ and $\widetilde{S}$ and separates them, as desired.

Otherwise, if $\mathcal T(\widetilde{S}^*\cap\widetilde{M}_o)$ intersects $\mathcal T(\widetilde{S}'\cap \widetilde{M}_o)\cup \mathcal T(\widetilde{S}\cap\widetilde{M}_o)$, we can assume without loss of generality that there is a JSJ or boundary plane $\widetilde{T}\in \mathcal T(\widetilde{S}^*\cap\widetilde{M}_o)\cap\mathcal T(\widetilde{S}'\cap\widetilde{M}_o)$. By the definition of $R'$ and $r=R+R'$, there is a translate $h\widetilde{S}'$ disjoint from $\widetilde{S}'$
such that $h\widetilde{S}'\cap \widetilde{T}$ separates $\widetilde{S}'\cap \widetilde{T}$ from $N_{R}(\widetilde{S}^*\cap \widetilde{T})$ in the intrinsic metric on $\widetilde{T}$.
Moreover, by Remark~\ref{rem:far->disjoint} or Lemma~\ref{lem:far->disjoint} the surface $h\widetilde{S}'\cap\widetilde{M}_o$ is disjoint from $\widetilde{S}^*\cap\widetilde{M}_o$ and $\mathcal T(h\widetilde{S}'\cap\widetilde{M}_o)$ intersects $\mathcal T(\widetilde{S}^*\cap\widetilde{M}_o)$ only in $\widetilde{T}$.
Hence $h\widetilde{S}'$ and $\widetilde{S}$ are disjoint and $h\widetilde{S}'$ separates $\widetilde{S}'$ from $\widetilde{S}$, as desired.
\end{proof}

\section{Separability in special cube complexes}
\label{sec:separability}

The goal of the next three sections is to prove Theorem~\ref{thm:specialization} (Specialization). We begin with reviewing the definition of a special cube complex.

\subsection{Special cube complexes}
\begin{defin}[{compare \cite[Def 3.2]{HW}}]
\label{def:special}
Let $X$ be a nonpositively curved cube complex, possibly not compact.
A \emph{midcube} (resp.\ \emph{codim--$2$--midcube}) of an $n$--cube $[0,1]^n=I^n$ is the subspace obtained by restricting exactly one (resp.\ two) coordinate to $\frac{1}{2}$. Let $\mathcal{M}$ denote the disjoint union of all
midcubes (resp.\ codim--$2$--midcubes) of $X$. An \emph{immersed hyperplane} (resp.\ \emph{immersed codim--$2$--hyperplane}) of $X$ is a connected component of the quotient of $\mathcal{M}$ by the inclusion maps.

An immersed hyperplane (resp.\ immersed codim--$2$--hyperplane) $A$ of $X$ \emph{self-intersects} if it contains two different midcubes (resp.\ codim--$2$--midcubes) of the same cube of $X$.
If $A$ does not self-intersect, then it embeds into $X$, and is called a \emph{hyperplane} (resp.\ \emph{codim--$2$--hyperplane}).
If the hyperplanes of $X$ do not self-intersect, which happens for example when $X$ is CAT(0), then codim--$2$--hyperplanes are components of intersections of pairs of intersecting hyperplanes. For an immersed hyperplane $A$, the map $A\rightarrow X$ is $\pi_1$--injective since it is a local isometry. We shall regard $\pi_1A$ as a subgroup of $\pi_1X$.

An edge $e$ is \emph{dual} to an immersed hyperplane $A$ if $A$ contains the midcube of $e$. A hyperplane $A$ is \emph{two-sided} if one can orient all of its dual edges so that any two that are parallel in a square $s$ of $X$ are oriented consistently within $s$.

If a hyperplane $A$ is two-sided and we orient its dual edges as above, we say that $A$ \emph{directly self-osculates}, if it has two dual edges with the same initial vertex or with the same terminal vertex. If $A$ is two-sided and the initial vertex of one of its dual edges coincides with the terminal vertex of another or the same dual edge, then $A$ \emph{indirectly self-osculates}.

Distinct hyperplanes $A,B$ \emph{interosculate}, if there are dual edges $e_1,e_2$ of $A$ and $f_1,f_2$ of $B$ such that $e_1,f_1$ lie in a square and $e_2,f_2$ share a vertex but do not lie in a square.

A nonpositively curved cube complex is \emph{special}, if its immersed hyperplanes do not self-intersect, are two-sided, do not directly self-osculate or inter-osculate.

A group is \emph{special} if it is the fundamental group of a special cube complex.
\end{defin}

Note that we do not require special cube complexes to be compact. However, in this article we will always assume that they have finitely many hyperplanes.

\begin{thm}[{\cite[Thm 4.2]{HW}}]
\label{thm:Artin}
A special cube complex $X$ with finitely many hyperplanes admits a local isometry $X\rightarrow R(X)$ into the Salvetti complex $R(X)$ of a finitely generated right-angled Artin group.
\end{thm}

The generators of the Artin group correspond to the hyperplanes of $X$. Each edge of $X$ dual to a hyperplane $A\subset X$ is mapped by the local isometry to an edge of $R(X)$ labeled by the generator corresponding to $A$. Note that $R(X)$ is compact special.

\smallskip

Our goal is to revisit and strengthen hyperplane separability and double hyperplane separability established in \cite{HW} for compact special cube complexes. The starting point and the main tool is the following.
\begin{thm}[{\cite[Cor 6.7]{HW}}]
\label{thm:canonical_completion_and_retraction}
Let $Y\rightarrow X$ be a local isometry from a compact cube complex $Y$ to a special cube complex $X$. There is a finite cover $\widehat{X}\rightarrow X$, called the \emph{canonical completion} of $Y\rightarrow X$, to which $Y$ lifts and a \emph{canonical retraction} map $\widehat{X}\rightarrow Y\subset \widehat{X}$, restricting to the identity on $Y$, which is continuous and maps hyperplanes of $\widehat{X}$ intersecting $Y$ into themselves.
\end{thm}

If one first subdivides $X$ (or takes an appropriate cover) to eliminate indirect self-osculations, then the canonical retraction can be made cellular.

\smallskip

All paths that we discuss in $X$ are assumed to be combinatorial. Let $\systole{X}$ denote the minimum of the lengths of essential closed paths in $X$ or $\infty$ if $X$ is contractible.

\begin{lem}
\label{lem:special_resfin}
Let $X$ be a special cube complex with finitely many hyperplanes. Then for each $d$ there is a finite cover $\widehat{X}$ of $X$ with $\systole{\widehat{X}}>d$.
\end{lem}

Note that the above property is preserved when passing to further covers.

\begin{proof}
Let $X\rightarrow R=R(X)$ be the local isometry into the Salvetti complex of the finitely generated right-angled Artin group $F$ coming from Theorem~\ref{thm:Artin}. Since $R$ is compact, there is a finite set $\mathcal{F}$ of conjugacy classes of elements of $F$ that can be represented by closed paths of length $\leq d$ in $R$. Since $F$ is residually finite, it has a finite index subgroup $\widehat{F}$ disjoint from the set of elements whose classes lie in $\mathcal{F}$. Let $\widehat{R}\rightarrow R$ be the finite cover corresponding to $\widehat{F}\subset F$. Then $\systole{\widehat{R}}>d$. Let $\widehat{X}\rightarrow X$ be the pullback of $\widehat{R}\rightarrow R$. Since $\widehat{X}\rightarrow \widehat{R}$ is a local isometry, it is $\pi_1$--injective and we have
$\systole{\widehat{X}}>d$ as desired.
\end{proof}

\subsection{Separability}

A subgroup $H$ of a group $G$ is \emph{separable} if for each $g\in G-H$, there is a finite index subgroup $F$ of $G$ with $g\notin FH$.

\begin{defin}
Let $X$ be a nonpositively curved cube complex and $\widetilde{X}$ its universal cover. Let $A$ be an immersed hyperplane in $X$ with an elevation $\widetilde{A}$ in $\widetilde{X}$.
The \emph{carrier} $N(\widetilde{A})$ is the smallest subcomplex of $\widetilde{X}$ containing $\widetilde{A}$. It is isomorphic with $\widetilde{A}\times I$. The \emph{carrier} $N(A)$ is the quotient of $N(\widetilde{A})$ by $\mathrm{Stab}(\widetilde{A})$. There is an induced map $N(A)\rightarrow X$. If $A$
does not self-intersect and does not self-osculate (directly or indirectly), then $N(A)$ embeds in $X$ and we identify it with its image. We similarly define carriers of immersed codim--$2$--hyperplanes.

A path $\alpha\rightarrow X$ \emph{starting (resp.\ ending) at a vertex $v$ of $N(A)$} is a path that starts (resp.\ ends) at the image of $v$ in $X$.
The path $\alpha$ is \emph{in $N(A)$} if it lifts to a path in $N(A)$.
The path $\alpha$ is path-homotopic \emph{into $N(A)$} if it is path-homotopic to a path in $N(A)$.
\end{defin}

\begin{defin}
An immersed hyperplane $A$ in a cube complex $X$ has \emph{injectivity radius} $>d$ if all paths of length $\leq 2d$ in $X$ starting and ending at $N(A)$ are path-homotopic into $N(A)$. In particular if $d=0$, then $A$ does not self-intersect or self-osculate.
Equivalently, all elevations of $N(A)$ to the universal cover $\widetilde{X}$ of $X$ are at distance $>2d$.
\end{defin}

\begin{lem}
\label{lem:special->inj_rad}
Let $X$ be a special cube complex with finitely many hyperplanes. Let $A\subset X$ be a hyperplane. Then for each $d$ there is a finite cover $\widehat{X}\rightarrow X$ such that any elevation $\widehat{A}\subset \widehat{X}$ of $A$ has injectivity radius $>d$.
\end{lem}

In the compact case, Lemma~\ref{lem:special->inj_rad} and the following consequence was proved in {\cite[Cor 9.7]{HW}} using Theorem~\ref{thm:canonical_completion_and_retraction}. Note that the conclusion of Lemma~\ref{lem:special->inj_rad} is preserved when passing to further covers.

\begin{cor}
\label{cor:special->separable}
Let $G$ be the fundamental group of a virtually special cube complex with finitely many hyperplanes and let $H\subset G$ be the fundamental group of an immersed hyperplane. Then $H$ is separable in $G$.
\end{cor}

\begin{proof}[Proof of Lemma~\ref{lem:special->inj_rad}]
As before, let $X\rightarrow R=R(X)$ be the local isometry into the Salvetti complex of the finitely generated right-angled Artin group $F$ coming from Theorem~\ref{thm:Artin}. Let $T$ be the hyperplane in $R$ that is the image of the hyperplane $A$. Since $R$ is compact, it admits finitely many paths starting and ending at $N(T)$ of length $\leq 2d$, not path-homotopic into $N(T)$. Let $\mathcal{F}$ denote the family of conjugacy classes determined by closing them up by paths in $N(T)$. Then $\mathcal{F}$ is a union of classes determined by finitely many nontrivial cosets of the form $Hg$, where $H=\pi_1T$. Since hyperplane subgroups in $F$ are separable \cite[Cor 9.4]{HW}, there is a finite index subgroup $\widehat{F}\subset F$ disjoint from the set of elements whose classes lie in $\mathcal{F}$. Let $\widehat{R}\rightarrow R$ be the finite cover corresponding to $\widehat{F}\subset F$. Then elevations of $T$ to $\widehat{R}$ have injectivity radius $>d$. Let $\widehat{X}\rightarrow X$ be the pullback of $\widehat{R}\rightarrow R$.

We verify that $\widehat{X}$ is the desired cover. The universal cover $\widetilde{X}$ of $X$ embeds into the universal cover $\widetilde{R}$ of $R$ as a convex subcomplex. Let $\widetilde{A}$ be an elevation of $A$ to $\widetilde{X}$ and let $\widetilde{T}$ be the elevation of $T$ to $\widetilde{R}$ containing $\widetilde{A}$. The $\pi_1 \widehat{X}$ orbit of $\widetilde{A}$ is contained in the $\pi_1 \widehat{R}$ orbit of a $\widetilde{T}$.
Since $\pi_1 \widehat{R}$ translates of $N(\widetilde{T})$ in $\widetilde{R}$ are at distance $>2d$, so are the $\pi_1 \widehat{X}$ translates of $N(\widetilde{A})$ in $\widetilde{X}$.
\end{proof}

\subsection{Double coset separability}
Let $H_1,H_2\subset G$ be subgroups of a group $G$. The \emph{double coset} $H_1H_2$ is \emph{separable}, if for each $g\in G- H_1H_2$ there is a finite index subgroup $F$ of $G$ with $g\notin FH_1H_2$.

\begin{defin}
Let $A$ be a hyperplane in a nonpositively curved cube complex $X$. Let $\widetilde{A}$ be an elevation of $A$ to the universal cover $\widetilde{X}$ of $X$. Let $\widetilde{A}^{+d}\subset \widetilde{X}$ be the combinatorial ball of radius $d$ around the carrier of $\widetilde{A}$.
We say that $A$ is \emph{$d$--locally finite}, if $\widetilde{A}^{+d}$ has finitely many $\mathrm{Stab(A)}$ orbits of hyperplanes.
\end{defin}

In particular, $A$ is $0$--locally finite if there are finitely many $\mathrm{Stab(A)}$ orbits of hyperplanes intersecting $\widetilde{A}$. If additionally there are finitely many $\mathrm{Stab(A)}$ orbits of hyperplanes osculating with $\widetilde{A}$, then $A$ is $1$--locally finite.

\begin{lemma}
\label{lem:special->double cosets separable}
Let $G$ be the fundamental group of a special cube complex with finitely many hyperplanes. Let $H_1,H_2\subset G$ be conjugates of the fundamental groups of hyperplanes one of which is $d$--locally finite for all $d$. Then the double coset $H_1H_2$ is separable in $G$.
\end{lemma}

While Lemma~\ref{lem:special->double cosets separable} could be avoided in the proof of Theorem~\ref{thm:main}, we include it to shed more light on double hyperplane separability.

\begin{proof}[Proof of Lemma~\ref{lem:special->double cosets separable}.]
Let $\widetilde{X}$ be the universal cover of the special cube complex $X$ with $\pi_1X=G$ and finitely many hyperplanes. Let $\widetilde{A},\widetilde{B}\subset \widetilde{X}$ be the hyperplanes stabilized by $H_1,H_2$. Let $A,B\subset X$ be the projections
of $\widetilde{A},\widetilde{B}$. Without loss of generality we may assume that $A$ is $d$--locally finite for all $d$. Let $\widetilde{v}$ be a base vertex of $N(\widetilde{A})$. Choose a path $\widetilde{\rho}\rightarrow \widetilde{X}$ starting at $\widetilde{v}$ and ending with an edge $\widetilde{e}$ dual to $\widetilde{B}$. Let $v, e$ be the projections of $\widetilde{v}, \widetilde{e}$ to $N(A),N(B)$. Then $\widetilde{\rho}$ projects to a path $\rho$ that starts at $v$ and ends with $e$. The elements of $H_1H_2$ are represented by closed paths of the form $\alpha \rho \beta \rho^{-1}$, where $\alpha,\beta$ are closed paths in $N(A),N(B)$ based at $v$ and the endpoint of $\rho$. Let $\gamma\rightarrow X$ be a closed path based at $v$ representing an element outside $H_1H_2$. We want to find a finite cover $\widehat{X}$ of $X$, where the based lifts of $\gamma$ and each path $\alpha \rho \beta \rho^{-1}$ above have distinct endpoints. Equivalently, we want the based lift of $\gamma \rho$ and each lift of $\rho$ starting at the preimage of $v$ in the based elevation of $N(A)$ to end with edges dual to distinct elevations of $B$. Here a \emph{based lift} or \emph{elevation} is a lift or elevation where $v$ lifts to a specified basepoint of $\widehat{X}$.

Suppose that $\rho$ and $\gamma\rho$ have length $\leq d$. By Lemma~\ref{lem:special->inj_rad} we can assume that $A$ has injectivity radius $>d$. Then the quotient $A^{+d}=H_1\backslash\widetilde{A}^{+d}$ embeds into $X$. Since $A$ is $d$--locally finite, there are finitely many hyperplanes in $A^{+d}$.
Applying Theorem~\ref{thm:Artin} to $A^{+d}$, let $A^{+d}\rightarrow R(A^{+d})$ be the local isometry into the Salvetti complex $R(A^{+d})$ of the right-angled Artin group with generators corresponding to hyperplanes in $A^{+d}$. Apply Theorem~\ref{thm:canonical_completion_and_retraction} to the induced local isometry $R(A^{+d})\rightarrow R(X)$. Consider its canonical completion $\widehat{R(X)}\rightarrow R(X)$ and the retraction $\widehat{R(X)}\rightarrow R(A^{+d})$. Take the pullback of the cover $\widehat{R(X)}\rightarrow R(X)$ to $\widehat{X}\rightarrow X$. We now verify that $\widehat{X}$ is the required cover.

Let $\widehat{A}\subset \widehat{X}$ be an elevation of $A$ mapping to $R(A^{+d})\subset \widehat{R(X)}$.
Let $\widehat{B}, \widehat{B}'\subset \widehat{X}$ be hyperplanes dual to the last edges $\widehat{e},\widehat{e}'\subset \widehat{A}^{+d}$ of lifts of $\rho, \gamma\rho$ starting at some lifts of $v$ in $N(\widehat{A})$. Since $\gamma$ represents an element outside $H_1H_2$, the hyperplanes in
$\widetilde{X}$ dual to the last edges of any lifts of $\rho, \gamma\rho$ starting at the $H_1$ orbit of $\widetilde{v}$ are distinct.
Hence the hyperplanes in $A^{+d}$ dual to the projections of $\widehat{e},\widehat{e}'$ are distinct.
Hence the projections of these hyperplanes to $R(A^{+d})$ are also distinct. The retraction $\widehat{R(X)}\rightarrow R(A^{+d})$ shows that hyperplanes $\widehat{T}_B, \widehat{T}'_B\subset\widehat{R(X)}$ containing these projections are also distinct. Since $\widehat{B}, \widehat{B}'$ map to $\widehat{T}_B, \widehat{T}'_B$, they are distinct as well.
\end{proof}

The proof of Lemma~\ref{lem:special->double cosets separable} also gives the following.

\begin{cor}
\label{cor:removing_intersection}
Let $X$ be a special cube complex with finitely many codim--$2$--hyperplanes. Let $A, B\subset X$ be hyperplanes and let $Q$ be a component of $A\cap B$. There is a finite cover $\widehat{X}\rightarrow X$ with the following property. If elevations $\widehat{A},\widehat{B}\subset \widehat X$ of $A,B$ intersect along an elevation of $Q$, then $\widehat A\cap \widehat B$ projects entirely to $Q$.
\end{cor}

\begin{proof}
Choose a component $Q'$ of $A\cap B$ distinct from $Q$. Let $e,e'$ be edges in $N(Q),N(Q')$ dual to $B$. Note that $e,e'$ are dual to distinct hyperplanes in $N(A)$. Consider closed paths $\gamma=\alpha\beta$ where $\alpha,\beta$ are paths in $N(A),N(B)$, and moreover $\alpha$ starts with $e$ and ends with $e'$. We need to find a cover $\widehat{X}$ where such paths $\gamma$ do not lift. In other words, the extremal edges of any lift of $\alpha$ to $\widehat{X}$ are dual to distinct elevations of $B$.

Since there are finitely many codim--$2$--hyperplanes in $X$, the carrier $N(A)$ has finitely many hyperplanes. In other words, the hyperplane $A$ is $0$--locally finite. We repeat the construction from the proof of Lemma~\ref{lem:special->double cosets separable} with $d=0$ to obtain the cover $\widehat{X}$. Let $\widehat A\subset \widehat X$ be an elevation of $A$ mapping to $R(N(A))\subset \widehat{R(X)}$. Then any lifts of $e,e'$ to $N(\widehat A)$ are dual to distinct hyperplanes in $\widehat X$. Replacing $\widehat{X}$ with a further cover that is a regular cover of $X$, we obtain the same property for all elevations of $A$.
\end{proof}

When we have a map $N(Q)\rightarrow N(A)$, a path in $N(A)$ \emph{starting (resp.\ ending) at a vertex $v$ of $N(Q)$} is a path that starts (resp.\ ends) at the image of $v$ in $N(A)$.

\begin{defin}
\label{def:double_injectivity}
Let $A\neq B$ be hyperplanes in a cube complex $X$ and let $\mathcal{Q}$ be a family of components of $A\cap B$. Hyperplanes $A,B$ have \emph{double injectivity radius} $>d$ at $\mathcal{Q}$ if all the paths of length $\leq 2d$ in $X$ starting at $N(A)$ and ending at $N(B)$ have the following property: They are path-homotopic to a concatenation at a vertex of $N(\mathcal{Q})$ of a pair of paths in $N(A)$ and $N(B)$. In particular $A\cap B=\mathcal{Q}$. In other words, if elevations $N(\widetilde{A}),N(\widetilde{B})$ of $N(A),N(B)$ to the universal cover of $X$ are at distance $\leq 2d$, then $\widetilde{A}\cap\widetilde{B}$ is nonempty and projects to $\mathcal{Q}$. We refer the reader to Figure~\ref{fig:Extra1}.
\end{defin}

\begin{figure}
\label{fig:Extra1}
\begin{center}
\includegraphics[width=.35\textwidth]{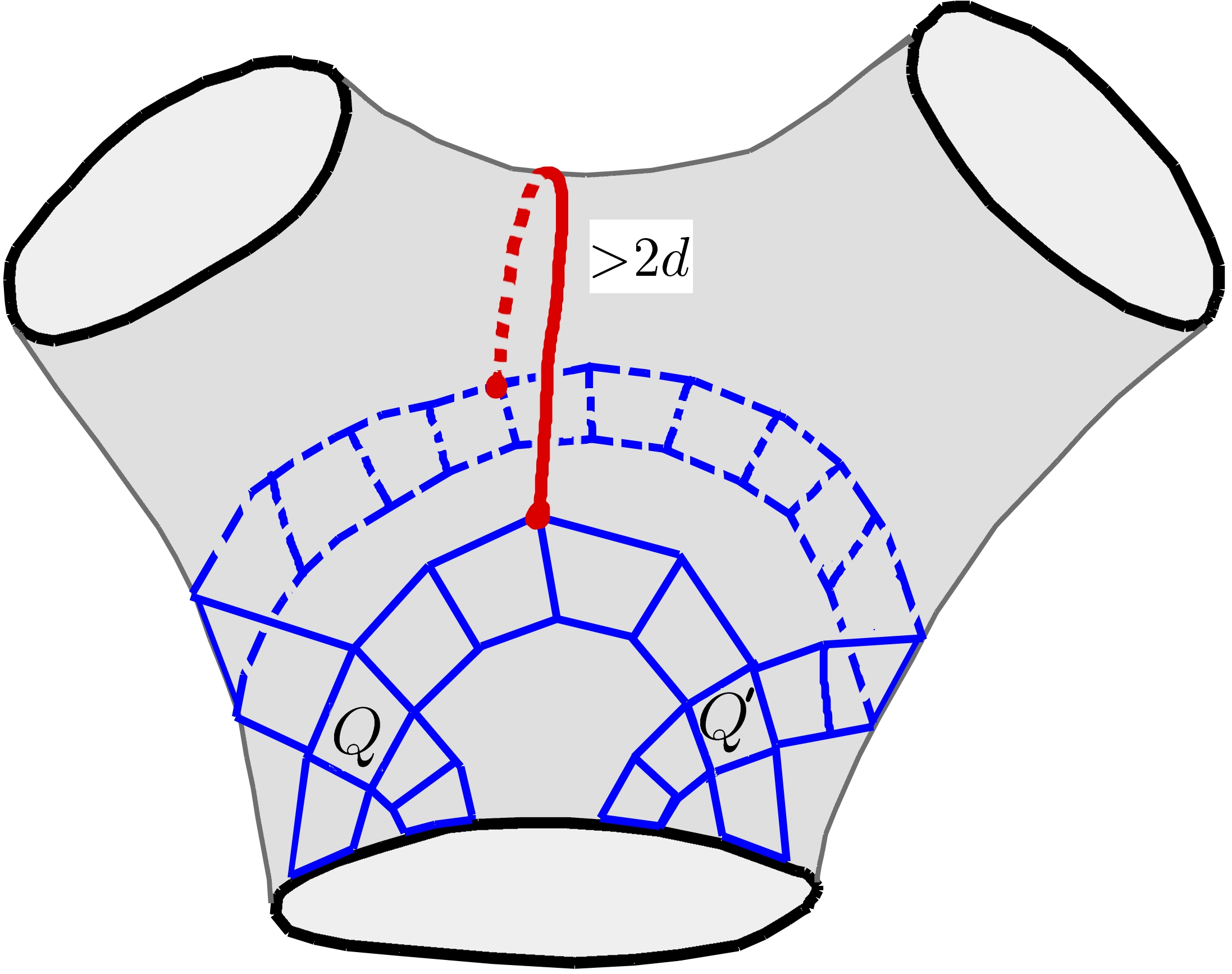}
\end{center}
\caption{Double injectivity radius $>d$ at $\mathcal Q=\{Q,Q'\}$.}
\label{fig:Extra1}
\end{figure}

\begin{lem}
\label{lem:special->double_inj_radius}
Let $X$ be a special cube complex with finitely many codim--$2$--hyperplanes. Let $A, B\subset X$ be hyperplanes and let $Q$ be a component of $A\cap B$. For each $d$ there is a finite cover $\widehat{X}\rightarrow X$ with the following property. If elevations $\widehat{A},\widehat{B}\subset \widehat{X}$ of $A,B$ intersect along an elevation of $Q$, then they have double injectivity radius $>d$ at the family of components of $\widehat{A}\cap \widehat{B}$ projecting to $Q$.
\end{lem}

Note that this property is preserved when passing to further covers. In particular, we can arrange that it holds for all $A,B$ and $Q$ simultaneously.

\begin{proof}
By Corollary~\ref{cor:removing_intersection} there is a finite cover $\widehat{X}$ of $X$ where the intersection of elevations $\widehat{A}\cap \widehat{B}$ is nonempty and projects to $Q$. Then passing to a regular cover and quotienting by the group permuting the components of $\widehat{A}\cap \widehat{B}$ reduces the situation to the case where $A\cap B$ is connected, i.e.\ $A\cap B=Q$.

Applying Theorem~\ref{thm:Artin}, let $X\rightarrow R$ be the local isometry into the Salvetti complex $R=R(X)$ of a finitely generated right-angled Artin group $F$. Let $T_A, T_B\subset R$ denote the hyperplanes that are the images of $A,B$ and let $T_Q\subset R$ denote the codim--$2$--hyperplane $T_A\cap T_B$. Consider paths $\gamma\rightarrow R$ of length $\leq 2d$ starting at $N(T_A)$ and ending at $N(T_B)$ but with $\gamma$ not path-homotopic to a concatenation at $N(T_Q)$ of a pair of paths in $N(T_A), N(T_B)$. Since $R$ is compact, there are finitely many such paths $\gamma$. Let $\mathcal{F}$ denote the family of conjugacy classes of elements of $F$ determined by the closed paths $\alpha\beta \gamma^{-1}$, with $\alpha$ in $N(T_A)$ and $\beta$ in $N(T_B)$ concatenated at $N(T_Q)$. Then $\mathcal{F}$ is a union of classes determined by finitely many nontrivial double cosets of the form $H_1H_2g$, where $H_1=\pi_1T_A,H_2=\pi_1T_B$. Since double cosets of hyperplane subgroups in $F$ are separable (a case of Lemma~\ref{lem:special->double cosets separable}, proved in \cite[Cor 9.4]{HW}), the group $F$ has a finite index subgroup $\widehat{F}$ disjoint from the set of elements whose classes lie in $\mathcal{F}$. Let $\widehat{R}\rightarrow R$ be the finite cover corresponding to $\widehat{F}\subset F$. Any intersecting elevations $\widehat{T}_A,\widehat{T}_B\subset \widehat{R}$ of $T_A,T_B$ have double injectivity radius $>d$ at $\widehat{T}_A\cap \widehat{T}_B$. Let $\widehat{X}\rightarrow X$ be the pullback of $\widehat{R}\rightarrow R$.

We show that $\widehat{X}$ has the desired property. Let $\widehat{A}, \widehat{B}\subset \widehat{X}$ be intersecting elevations of $A,B$. Let $\widetilde{A}, \widetilde{B}$ be their further elevations to the universal cover $\widetilde{X}$ of $X$ at distance $\leq 2d$. The universal cover $\widetilde{X}$ embeds as a convex subcomplex of the universal cover $\widetilde{R}$ of $R$.
The hyperplanes $\widetilde{T}_A, \widetilde{T}_B\subset \widetilde{R}$ containing $\widetilde{A}, \widetilde{B}$ intersect, since their images $\widehat{T}_A,\widehat{T}_B\subset \widehat{R}$ have double injectivity radius $>d$ at $\widehat{T}_A\cap \widehat{T}_B$. By Helly's theorem \cite[Thm~2.2]{Rol} the combinatorial convex hull of a pair points in intersecting hyperplanes contains an intersection point. Hence the hyperplanes $\widetilde{A}$ and $\widetilde{B}$ intersect as well.
\end{proof}

\section{Background on cubical small cancellation}
\label{sec:small_cancellation}
In this section we review the main theorem of cubical small cancellation \cite{Hier}. It will be used in the proof of Theorem~\ref{thm:specialization} (Specialization).

\subsection{Pieces}
Let $X$ be a nonpositively curved cube complex. Let $\{Y_i\rightarrow X\}$ be a collection of local isometries of nonpositively curved
cube complexes. The pair $\langle X|\{Y_i\rightarrow X\}\rangle$, or briefly $\langle X|Y_i\rangle$, is a \emph{cubical presentation}.
Its \emph{group} is $\pi_1X/ \nclose{\{\pi_1Y_i\}}$ which equals $\pi_1X^*$ where $X^*$ is obtained from $X$ by attaching cones along the $Y_i$. Let $\overline{X}=\nclose{\{\pi_1Y_i\}}\backslash \widetilde{X}$ denote the cover of $X$ in the universal cover $\widetilde{X}^*$ of $X^*$.

An \emph{abstract cone-piece in $Y_i$} of $Y_j$ is the intersection $P=\widetilde{Y}_i\cap\widetilde{Y}'_j$ of some elevations $\widetilde{Y}_i, \widetilde{Y}'_j$ of $Y_i,Y_j$ to the universal cover $\widetilde{X}$ of $X$. In the case where $j=i$
we require that the elevations are distinct in
the sense that for the projections $P\rightarrow Y_i,Y_j$ there is no automorphism $Y_i\rightarrow Y_j$ such that the following diagram commutes:
$$
\begin{array}{ccc}

  P & \rightarrow & Y_i \\

  \downarrow& \swarrow & \downarrow \\

  Y_j & \rightarrow & X

\end{array}
$$
Note that an abstract cone-piece \emph{in $Y_i$} actually lies in $\widetilde{Y}_i$.

Let $\widetilde{A}$ be a hyperplane in $\widetilde{X}$ disjoint from $\widetilde{Y}_i$.
An \emph{abstract wall-piece in $Y_i$} is the intersection $\widetilde{Y}_i\cap N(\widetilde{A})$.
An \emph{abstract piece} is an abstract cone-piece or an abstract wall-piece.

A path $\alpha\rightarrow Y_i$ is a \emph{piece in $Y_i$}, if it lifts to $\widetilde{Y}_i$ into an abstract piece in $Y_i$.
We then denote by $|\alpha|_{Y_i}$ the combinatorial distance between the endpoints of a lift of $\alpha$ to $\widetilde{Y}_i$, i.e.\ the length of a geodesic path in $Y_i$ path-homotopic to $\alpha$.

The cubical presentation $\langle X|Y_i\rangle$ satisfies the \emph{$C'(\frac{1}{n})$ small cancellation condition},
if $|\alpha|_{Y_i} < \frac{1}{n}\systole{Y_i}$ for each piece $\alpha$ in $Y_i$.
Recall that $\systole{Y_i}$ denotes the minimum of the lengths of essential closed paths in $Y_i$.

\subsection{Ladder Theorem}

A \emph{disc diagram} $D$ is a compact contractible $2$--complex with a fixed embedding in $\R^2$. Its \emph{boundary path} $\partial_\p D$ is the attaching map of the cell at $\infty$. The diagram is \emph{spurless} if $D$ does not have a \emph{spur}, i.e.\ a vertex contained in only one edge. If $X$ is a combinatorial complex, a \emph{disc diagram in $X$} is a combinatorial map of a disc diagram into $X$.

Let $D\rightarrow \widetilde{X}^*$ be a disc diagram with a boundary path $\partial_\p D\rightarrow \overline{X}$. Note that the $2$--cells of $\widetilde{X}^*$ are squares or triangles, where the latter have exactly one vertex at a cone point. The triangles in $D$ are grouped together into \emph{cone-cells} around these cone points. The \emph{complexity} of $D$ is the pair of numbers $(\#$ cone-cells of $D,\  \# $ squares of $D)$, with lexicographic order.

In addition to spurs, there are two other types of positive curvature features in $\partial_\p D$: \emph{shells} and \emph{cornsquares}. A cone-cell $C$ adjacent to $\partial_\p D$ is a \emph{shell} if $\partial C\cap \partial_\p D$ (\emph{outer path}) is connected and its complement in $\partial C$ (\emph{inner path}) is a concatenation of $\leq 6$ pieces. A pair of consecutive edges of $\partial_\p D$ is a \emph{cornsquare} if the
carriers of their dual hyperplanes intersect at a square and surround a square subdiagram, i.e.\ a subdiagram all of whose $2$--cells are squares.
A \emph{ladder} is a disc diagram that is the concatenation of cone-cells and rectangles with cone-cells or spurs at extremities, as
in Figure~\ref{fig:SmallCancellationThings}. A single cone-cell is not a ladder, while a single edge is a ladder. The following summarizes the main results of cubical small cancellation theory.

\begin{thm}[{\cite[Thm 3.40]{Hier}}]
\label{thm:ladder}
Assume that $\langle X|Y_i\rangle$ satisfies the $C'(\frac{1}{12})$ small cancellation condition. Let $D\rightarrow \widetilde{X}^*$ be a minimal complexity disc diagram for a closed path $\partial_\p D\rightarrow \overline{X}$. Then one of the following holds.
\begin{enumerate}[(a)]
\item $D$ is a single vertex or a single cone-cell.
\item $D$ is a ladder.
\item $D$ has at least three spurs and/or shells and/or cornsquares. Moreover, if there is no shell or spur, then there must be at least four cornsquares.
\end{enumerate}
\end{thm}

\begin{figure}
\begin{center}
\includegraphics[width=\textwidth]{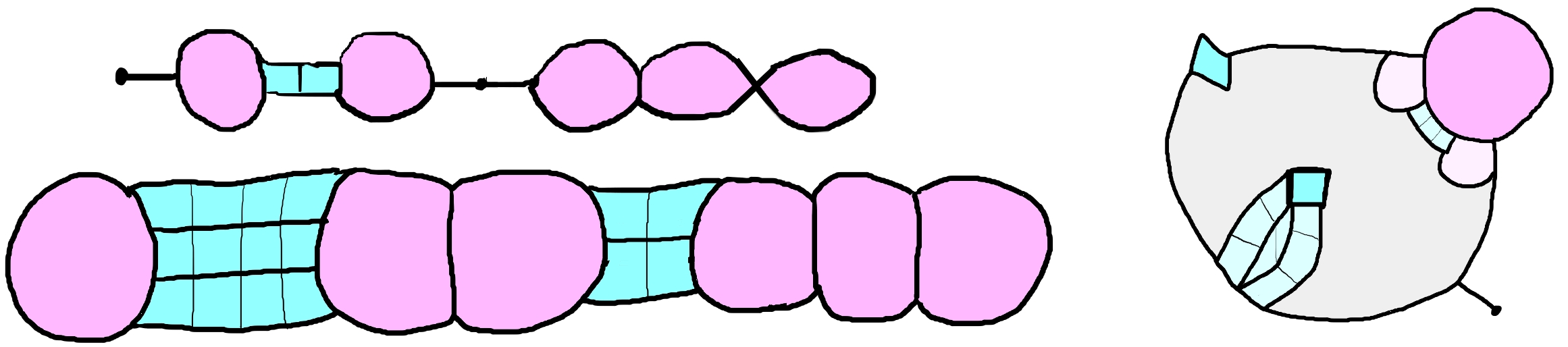}
\end{center}
\caption{Two ladders are on the left. On the right are two cornsquares, a spur, and a shell within a disc diagram.
\label{fig:SmallCancellationThings}
}
\end{figure}

The following consequence allows us to identify $Y_i$ with any of its lifts to $\overline{X}$.

\begin{cor}
\label{cor:Yembed}
Let $\langle X|Y_i\rangle$ satisfy the $C'(\frac{1}{12})$ small cancellation condition. Then each $Y_i$ lifts to an embedding in $\overline{X}$.
\end{cor}
\begin{proof}
We argue by contradiction. Let $\gamma\rightarrow Y_i$ be a path of minimal length that is not a closed path but projects to a closed path $\overline{\gamma}\rightarrow\overline{X}$. Let $\overline{v}$ be the vertex of $\overline{\gamma}$ which is the projection of the endpoints of $\gamma$.
Let $D\rightarrow \overline{X}$ be a disc diagram with $\partial_\p D=\overline{\gamma}$ of minimal complexity among all such paths $\gamma$. Then the boundary $\partial C$ of any cone-cell $C$ in $D$ is essential in $Y_j$ into which it maps.

If in $\overline \gamma-\overline v$ there were two consecutive edges forming a cornsquare, we could homotope $D$ so that there is a square at that exact corner \cite[Lem 2.6]{Hier}. That square would lift to $Y_i$ and we could homotope $\gamma$ through it to reduce the complexity. The diagram $D$ has no spur. If there is a shell $C$ in $D$ whose outer path is contained in $\overline \gamma-\overline v$, then let $\gamma_{\ds C}=\overline\gamma\cap\partial C$ denote that outer path and suppose that $\partial C$ maps to $Y_j$. If $\gamma_{\ds C}$ is a piece in $Y_j$ of $Y_i$, then this contradicts that $\partial C$ is essential in $Y_i$. Otherwise $j=i$ and there is an identification $Y_j\rightarrow Y_i$ agreeing on $\gamma_{\ds C}$. We then replace inside $\gamma$ the outer path $\gamma_{\ds C}$ by the inner path of $C$ to obtain $\gamma'\rightarrow Y_i$ with the same endpoints as $\gamma$. The projection $\overline{\gamma}'\rightarrow \overline X$ of $\gamma'$ bounds the disc diagram $D-C$ of smaller complexity than $D$, which is a contradiction. Hence by Theorem~\ref{thm:ladder}, the diagram $D$ is a single cone-cell $C$. Since $\partial C$ is essential, the path $\overline\gamma-\overline v$ is not a piece, hence again we can identify $Y_j$ with $Y_i$ along $\overline\gamma-\overline v$ and $\gamma$. But $\overline\gamma-\overline v$ is a closed path in $Y_j$, contradiction.
\end{proof}

\subsection{Small cancellation quotients}
We now prove that in small cancellation quotients we can separate elements from cosets and double cosets. We also prove a convexity result for `extended carriers'.

\begin{lemma}
\label{lem:quasiconvex}
Let $\langle X|Y_i\rangle$ be a cubical presentation with all abstract pieces of uniformly bounded diameter.
Suppose that each $Y_i$ is virtually special with finitely many immersed hyperplanes.
Let $\widetilde{A}\subset \widetilde{X}$ be a hyperplane and let $g\in G-H$, where $G=\pi_1X$ and $H=\mathrm{Stab}(\widetilde{A})$. Then there are finite index subgroups $P_i'\subset P_i=\pi_1Y_i$ such that:
\begin{enumerate}[(1)]
\item
Letting $\overline{X}= \nclose{\{P_i'\}}\backslash\widetilde{X}$, the immersed hyperplane $\overline{A}$ in $\overline{X}$ that is the projection of $\widetilde A$ has no self-intersections and no self-osculations.
\item
Any two points of $N(\overline{A})$ are connected by a geodesic that lies in the union of $N(\overline{A})$ and the translates of $\overline{Y}_i=P'_i \backslash \widetilde{Y}_i$ in $\overline{X}$ intersecting $\overline{A}$.
\item
$\overline{g}\notin\overline{H}$ in the quotient $\overline{G}=G/ \nclose{\{P_i'\}}$.
\end{enumerate}
\end{lemma}

\begin{proof}
Assume that all the abstract pieces have diameter $<d$.
By Lemma~\ref{lem:special_resfin} we can choose $P'_i$ so that $\systole{\overline{Y}_i}\geq 12d$. By Lemma~\ref{lem:special->inj_rad}, we can further choose $P_i'$ so that all the hyperplanes of $\overline{Y}_i$ have injectivity radius $>3d$. We also require that $P_i'\subset P_i$ are characteristic, so that $\langle X|\overline{Y}_i\rangle$ satisfies the $C'(\frac{1}{12})$ small cancellation condition. Note that merely requiring that $P_i'$ is normal in $P_i$ might not suffice, since we need that every automorphism of $Y_i$ respecting the map $Y_i\rightarrow X$ elevates to $\overline{Y}_i$, thus ensuring that in $\widetilde{Y}_i$ there do not appear new abstract cone-pieces in $\overline{Y}_i$ of $\overline{Y}_i$.

We first prove Assertion (2). Let $D\rightarrow\widetilde{X}^*$ be a disc diagram bounded by a geodesic $\alpha$ in the $1$--skeleton of $N(\overline{A})$ and a geodesic $\gamma$ in the $1$--skeleton of $\overline{X}$. We assume that $D$ has minimal complexity among all such disc diagrams with prescribed common endpoints of $\alpha$ and $\gamma$. Then the boundary $\partial C$ of any cone-cell $C$ in $D$ is essential in $\overline{Y}_i$ into which it maps. Hence
$|\partial C|_{\overline{Y}_i}\geq \systole{\overline{Y}_i}\geq 12d$, where $|\partial C|_{\overline{Y}_i}$ denotes the minimal length of a closed path in the free homotopy class of $\partial C$ in $\overline{Y}_i$. Consequently, since the inner path of a shell $C$ is a concatenation of at most $6$ pieces, a geodesic in $\overline{Y}_i$ that is path-homotopic to the inner path of $C$ is shorter than the outer path of $C$.

If in $\partial_\p D= \alpha\gamma$ there are two consecutive edges forming a cornsquare, they cannot both lie in $\alpha$ or both lie in $\gamma$. Otherwise we could homotope $D$ so that there is a square at that exact corner \cite[Lem 2.6]{Hier}. Then we could homotope $\alpha$ or $\gamma$ through that square to reduce the complexity. The diagram $D$ has no spur except possibly where $\alpha$ and $\gamma$ are concatenated.
If there is a shell $C$ in $D$ whose outer path is contained in $\gamma$, then replacing the outer path of $C$ by a geodesic that is path-homotopic to the inner path of $C$ contradicts that $\gamma$ is a geodesic.

Finally, suppose that the outer path of a shell $C$ is contained in $\alpha$. Let $\alpha_{\ds C}=\partial C\cap\alpha$ denote the outer path of $C$ and let $\delta$ denote the inner path of $C$. Let $\widetilde{Y}_i$ denote the universal cover of $\overline{Y}_i$ into which $\partial C$ maps. Consider the copy of $\widetilde{Y}_i$ in $\widetilde{X}$ that contains a lift of $\alpha_{\ds C}$ to $N(\widetilde{A})$. If $\widetilde{A}$ is disjoint from $\widetilde{Y}_i$, then $\alpha_{\ds C}$ is a piece and $\partial C$ is a concatenation of at most $7$ pieces, which contradicts $|\partial C|_{\overline{Y}_i}\geq 12d$.
Otherwise let $\widetilde{A}_i=\widetilde{A}\cap\widetilde{Y}_i$. Hence $\alpha_{\ds C}$ projects into the quotient $N(\overline{A}_i)$ of $N(\widetilde{A}_i)$ in $\overline{Y}_i$. Since the injectivity radius of the hyperplane $\overline{A}_i$ in $\overline{Y}_i$ is $>3d$, the inner path $\delta$ is path-homotopic in $\overline{Y}_i$ to a path $\alpha'_C$ in $N(\overline{A}_i)$. If we choose $\alpha'_C$ to be geodesic, then since it is path-homotopic to the inner path $\delta$, we have
$|\alpha'_C|=|\delta|_{\overline{Y}_i}<|\alpha_{\ds C}|$. This contradicts that $\alpha$ is a geodesic in $N(\overline{A})$.

Thus there can be at most two spurs and/or shells and/or cornsquares in $D$ and these are located where $\alpha$ and $\gamma$ are concatenated.
By Theorem~\ref{thm:ladder}, the disc diagram $D$ is a single cone-cell or ladder. For any of its cone-cells $C$ let $\alpha_{\ds C}=\alpha\cap\partial C, \ \gamma_{\ds C}=\gamma\cap\partial C$. Let $\lambda_{\ds C},\delta_{\ds C}$ denote the remaining, possibly trivial, arcs of $\partial C$. Since $\lambda_{\ds C}, \delta_{\ds C}$ are pieces, we have $|\lambda_{\ds C}|_{\overline{Y}_i}<d$ and $|\delta_{\ds C}|_{\overline{Y}_i}<d$, where $\partial C$ maps to $\overline{Y}_i$. As before, if $\widetilde{A}$ is disjoint from $\widetilde{Y}_i$ containing a lift of $\alpha_{\ds C}$ to $N(\widetilde{A})$, then $\alpha_{\ds C}$ is a piece. Since $|\partial C|_{\overline{Y}_i}\geq 12d$, this contradicts that $\gamma_{\ds C}$ is a geodesic. Hence $\widetilde{A}$ intersects $\widetilde{Y}_i$, which proves Assertion (2).

\begin{figure}
\label{fig:Extra2}
\begin{center}
\includegraphics[width=.5\textwidth]{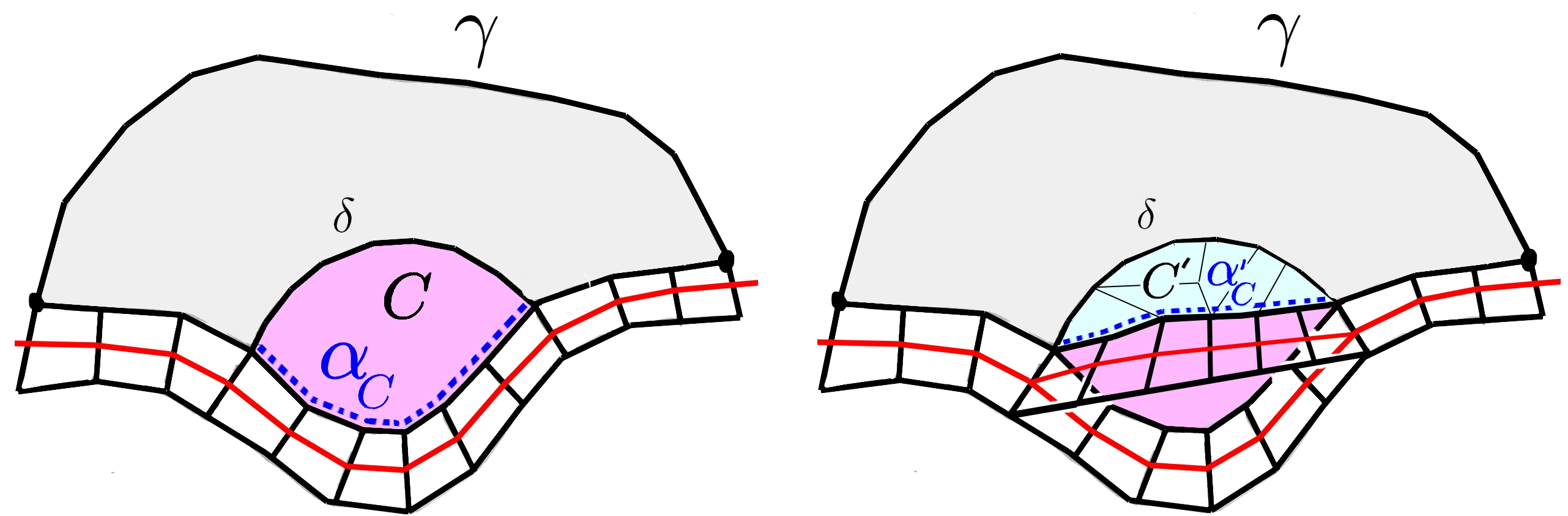}
\end{center}
\caption{The shell $C$ is surrounded by a hyperplane and a short inner path $\delta$.
We can thus replace $C$ by a square diagram as on the right.}
\label{fig:Extra2}
\end{figure}

For Assertion (3), let $A$ be the immersed hyperplane in $X$ that is the projection of $\widetilde{A}$. Let $\gamma\rightarrow X$ be a minimal length path starting and ending at $N(A)$ such that a concatenation of $\gamma$ with a path in $N(A)$ represents the conjugacy class of $g$. We increase $d$ so that $d\geq|\gamma|$, and we then choose $P_i'$ as before. If $\overline{g}$ lies in $\overline{H}$, then a lift of $\gamma$ to $\overline{X}$ forms a closed path with a path $\alpha$ in $N(\overline{A})$. Let $D$ be a minimal complexity disc diagram with $\partial_\p D=\alpha\gamma$ among all such $\gamma$. Then $D$ cannot have spurs by the minimality of $|\gamma|$. As before, there are no consecutive edges forming cornsquares in $\alpha$ or in $\gamma$. Since $d\geq|\gamma|$, there are no shells with outer path in $\gamma$.

If there is a shell $C$ with outer path $\alpha_{\ds C}=\alpha\cap \partial C$, then as before the inner path $\delta$ is path-homotopic in $\overline{Y}_i$ to a path $\alpha'_C$ in $N(\overline{A}_i)$. Then as in Figure~\ref{fig:Extra2}, we could replace $\alpha_{\ds C}$ by $\alpha'_C$ and replace $C$ by a square diagram. This contradicts the minimal complexity of $D$.

By Theorem~\ref{thm:ladder}, the disc diagram $D$ is a single cone-cell or a ladder and for any of its cone-cells $C$
the hyperplane $\widetilde{A}$ intersects $\widetilde{Y}_i$ containing the lift of $\alpha_{\ds C}$ to $N(\widetilde{A})$. We have $|\lambda_{\ds C}|_{\overline{Y}_i}+|\gamma_{\ds C}|+|\delta_{\ds C}|_{\overline{Y}_i}\leq 3d$, while the injectivity radius in $\overline{Y}_i$ of the projection of the hyperplane $\widetilde{A}_i=\widetilde{A}\cap\widetilde{Y}_i$ is $>3d$. Then we can replace $\alpha_{\ds C}$ by $\alpha'_C$ and replace $C$ by a square diagram to contradict the minimal complexity of $D$.

Assertion (1) follows from the same proof as Assertion (3), where we consider all paths $\gamma$ of length $0$.
\end{proof}

The following clarifies and generalizes \cite[Thm 16.23]{Hier}.

\begin{lemma}
\label{lem:quotient_double_separable}
Let $\langle X|Y_i\rangle$ be a cubical presentation with all abstract pieces of uniformly bounded diameter.
Suppose that each $Y_i$ is virtually special with finitely many immersed codim--$2$--hyperplanes.
Let $H_1,H_2\subset G=\pi_1X$ be stabilizers of intersecting hyperplanes $\widetilde{A},\widetilde{B}\subset \widetilde{X}$ and let $g\in G-H_1H_2$. There are finite index subgroups $P_i'\subset P_i=\pi_1Y_i$ such that $\overline{g}\notin\overline{H}_1\overline{H}_2$ in the quotient $\overline{G}=G/\nclose{\{P_i'\}}$.
\end{lemma}

\begin{proof}
Assume that all the abstract pieces have diameter $<d$.
Let $\widetilde{Q}=\widetilde{A}\cap \widetilde{B}$ and let $A,B,Q$ be the immersed hyperplanes and codim--$2$--hyperplane in $X$ that are the projections of $\widetilde A,\widetilde B,\widetilde Q$. Let $\gamma\rightarrow X$ be a minimal length path starting at $N(A)$ and ending at $N(B)$ such that its concatenation with a pair of paths in $N(A),N(B)$ concatenated at $N(Q)$ represents the conjugacy class of $g$. We increase $d$ so that $d\geq|\gamma|$. Let $P_i^{\mathrm{sp}}\subset P_i$ be finite index special subgroups.

By Lemmas~\ref{lem:special_resfin},~\ref{lem:special->inj_rad}, and~\ref{lem:special->double_inj_radius} we can choose finite index subgroups $P_i'\subset P_i^{\mathrm{sp}}$ that are characteristic in $P_i$, and such that $\overline{Y}_i=P'_i \backslash \widetilde{Y}_i$ satisfy:
\begin{itemize}
\item
$\systole{\overline{Y}_i}\geq 12d$,
\item
all hyperplanes in $\overline{Y}_i$ have injectivity radius $>4d$, and
\item
all pairs $\overline{A}_i,\overline{B}_i$ of hyperplanes in $\overline{Y}_i$ intersecting at a codim--$2$--hyperplane $\overline{Q}_i$ have double injectivity radius $>3d$ at the family of components of $\overline{A}_i\cap\overline{B}_i$ in the $P_i^{\mathrm{sp}}/{P'_i}$ orbit of $\overline{Q}_i$.
\end{itemize}

The reason we used the $P_i^{\mathrm{sp}}/{P'_i}$ orbit instead of the entire $P_i/{P_i'}$ orbit is the following.
Since $P_i^{\mathrm{sp}}$ is special, an element $p\in P_i^{\mathrm{sp}}/{P'_i}$ cannot map $\overline{A}_i$ to a distinct hyperplane intersecting $\overline{A}_i$.
Hence if $p\in P_i^{\mathrm{sp}}/{P'_i}$ maps a component of $\overline{A}_i\cap\overline{B}_i$ to a component of $\overline{A}_i\cap\overline{B}_i$, then it cannot interchange $\overline{A}_i$ and $\overline{B}_i$ and so it stabilizes $\overline{A}_i$ and $\overline{B}_i$.

Let $\overline{X}= \nclose{\{P_i'\}}\backslash\widetilde{X}$. Let $\overline{A},\overline{B},\overline{Q}$ be the hyperplanes and codim--$2$--hyperplane in $\overline{X}$ that are the projections of $\widetilde A,\widetilde B,\widetilde Q$.

We now argue by contradiction to prove the lemma. If $\overline{g}$ lies in $\overline{H}_1\overline{H}_2$, then there is a disc diagram $D\rightarrow\widetilde{X}^*$ bounded by a closed path $\alpha\gamma\beta^{-1}$, where $\alpha, \beta^{-1}$ are paths in $N(\overline{A}),N(\overline{B})$ concatenated at a vertex $\overline{v}\in N(\overline{Q})$, and we lift $\gamma$ to $\overline{X}$. Assume that $D$ has minimal complexity among all such diagrams and $\gamma$. By minimality of $|\gamma|$ the diagram $D$ has no spurs except possibly where $\alpha$ and $\beta$ are concatenated. By replacing $\overline{v}$ we can remove such spurs and assume that $D$ is spurless. The diagram $D$ also cannot have two consecutive edges of $\alpha, \beta$ or $\gamma$ forming cornsquares.

An outer path of a shell $C$ cannot be contained entirely in $\alpha, \beta$ or $\gamma$, as in the proof of Lemma~\ref{lem:quasiconvex}. We now prove that the outer path of a shell $C$ with $\partial C$ mapping to $\overline{Y}_i$ cannot be contained in $\alpha\gamma$ (or $\gamma\beta^{-1}$). Otherwise, recall that $|\gamma|\leq d$ and the length of a geodesic that is path-homotopic to the inner path of $C$ is $<6d$, hence if $\alpha_{\ds C}=\partial C\cap\alpha$ is a piece, then this contradicts $|\partial C|_{\overline{Y}_i}\geq 12d$. If $\alpha_{\ds C}$ is not a piece, then since the hyperplane injectivity radius in $\overline{Y}_i$ is $>4d$, we could replace $\alpha_{\ds C}$ by $\alpha'_C$ and replace $C$ by a square diagram, contradicting minimal complexity.

Since $\langle X|\overline{Y}_i\rangle$ satisfies the $C'(\frac{1}{12})$ small cancellation condition, by Theorem~\ref{thm:ladder} the disc diagram $D$ is either:
\begin{enumerate}[(a)]
\item
a single cone-cell $C$, or
\item
a ladder with a shell $C$ containing $\overline{v}$, or
\item
a diagram with 2 cornsquares located where $\gamma$ is concatenated with $\alpha$ and $\beta$, and with a shell $C$ containing $\overline{v}$ as in Figure~\ref{fig:TheFish}.
\end{enumerate}
In each case there is a $2$-cell $C$ containing $\overline{v}$. Let $\alpha_{\ds C}$ and $\beta_{\ds C}$ denote the subpaths $\alpha\cap \partial C$ and $\beta\cap \partial C$. The complement $\delta$ in $\partial C$ of $\beta_C^{-1}\alpha_{\ds C}$ either coincides with $\gamma$ in Case~(a), or is a piece in Case~(b) or is an inner path of a shell, hence a concatenation of at most $6$ pieces in Case~(c). In each case we have $|\delta|_{\overline{Y}_i}<6d$, where $\partial C$ maps to $\overline{Y}_i$.

Let $\widetilde{v}\in N(\widetilde{Q})$ be a lift of $\overline{v}$. Let $\widetilde{Y}_i$ be the elevation of $Y_i$ containing $\widetilde{v}$. Let $\widetilde{A}_i= \widetilde{A}\cap\widetilde{Y}_i,\ \widetilde{B}_i= \widetilde{B}\cap\widetilde{Y}_i$. If both $\widetilde{A}_i,\widetilde{B}_i$ are empty, then both $\alpha_{\ds C},\beta_{\ds C}$ are pieces which contradicts $|\partial C|_{\overline{Y}_i}>12d$. If exactly one of $\widetilde{A}_i,\widetilde{B}_i$ is empty, say $\widetilde{B}_i$, then $\beta_{\ds C}$ is a piece. Since the injectivity radius of $\widetilde{A}_i$ is $>4d$, as before we could replace $\alpha_{\ds C}$ by $\alpha'_C$ and replace $C$ by a square diagram, contradicting minimal complexity. Hence both $\widetilde{A}_i,\widetilde{B}_i$ are nonempty and $D$ shows that they intersect in nonempty $\widetilde{Q}_i=\widetilde{Q}\cap \widetilde{Y}_i$.

\begin{figure}
\begin{center}
\includegraphics[width=\textwidth]{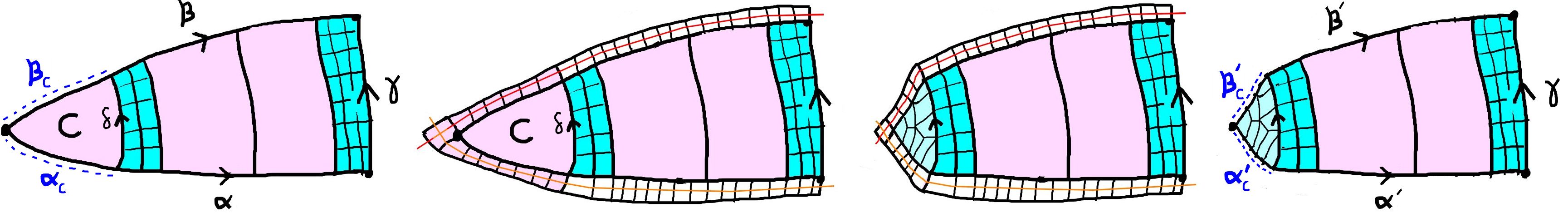}
\end{center}
\caption{The shell $C$ in the first diagram
is surrounded by two hyperplanes and a short inner path $\delta$ as in the second
diagram. We can thus replace $C$ by a square diagram bounded by $\alpha'_C \delta \beta'^{-1}_C$,
to obtain a smaller complexity diagram on the right.}
\label{fig:TheFish}
\end{figure}

Let $\overline{A}_i,\overline{B}_i,\overline{Q}_i\subset \overline{Y}_i$ denote the projections of $\widetilde{A}_i,\widetilde{B}_i,\widetilde{Q}_i$.
The double injectivity radius in $\overline{Y}_i$ is $>3d$. Hence $\delta$ is homotopic in $\overline{Y}_i$ to a concatenation at $pN(\overline{Q}_i)$ of paths $\alpha'_C,\beta'_C$ in $N(\overline{A}_i), N(\overline{B}_i)$ for some $p\in P^{\mathrm{sp}}_i/{P_i'}$.
Thus there is in $\overline{Y}_i$ a square diagram with boundary $\alpha'_C \delta {\beta'}_C^{-1}$.
We replace $C$ by this square diagram, and replace the subpath $\alpha_{\ds C}$
of $\alpha$ by $\alpha'_C$ to obtain $\alpha'$, and similarly we replace the subpath $\beta_{\ds C}$
of $\beta$ by $\beta'_C$ to obtain $\beta'$.
Since $p\in P^{\mathrm{sp}}_i/{P_i'}$, we have $pN(\overline{A}_i)=N(\overline{A}_i)$ and $pN(\overline{B}_i)=N(\overline{B}_i)$.
Translating the whole diagram by $p^{-1}$ yields a disc diagram bounded by $p^{-1}(\alpha')p^{-1}(\gamma)p^{-1}(\beta'^{-1})$, where $p^{-1}(\alpha'),p^{-1}(\beta'^{-1})$ are paths in $N(\overline{A}),N(\overline{B})$ concatenated at $N(\overline{Q})$. This diagram has a smaller number of cone-cells than $D$, which contradicts the minimal complexity assumption. See Figure~\ref{fig:TheFish}.
\end{proof}

\section{Specialization}
\label{sec:special}
In this section we prove Theorem~\ref{thm:specialization} (Specialization).

\begin{proof}[Proof of Theorem~\ref{thm:specialization}]
To prove that the action of $G$ on $\widetilde{X}$ is virtually special, we will verify the conditions of Criterion~\ref{thm:HWcriterion}. Freeness and finiteness Conditions (1)--(3) of Criterion~\ref{thm:HWcriterion} are Hypothesis~(i) of Theorem~\ref{thm:specialization}. We now verify Condition~(4). Let $H$ be the stabilizer of a hyperplane $\widetilde{A}\subset\widetilde{X}$. Let $g\in G-H$. We will find finite index subgroups $P'_i\subset P_i$ such that:

\begin{enumerate}[(a)]
\item
$\overline{G}=G/\nclose{\{P_i'\}}$ is hyperbolic and virtually compact special,
\item
the image $\overline{H}$ is quasiconvex in $\overline{G}$,
\item
$\overline{g}\notin\overline{H}$.
\end{enumerate}
The result then follows from separability of quasiconvex subgroups in hyperbolic virtually compact special groups \cite[Thm 7.3]{HW}.

By Hypothesis~(iv) and Theorem~\ref{thm:relative special qoutient}, there are $E^\circ_n\subset E_n$ such that $P'_i\cap E_n\ \subset E^\circ_n$ implies that $\overline{G}$ splits as a graph of hyperbolic virtually compact special groups with finite edge groups. Then $\overline{G}$ is hyperbolic virtually compact special and Condition~(a) is satisfied. By Hypothesis~(ii), there are indeed finite index subgroups $P'_i\subset P_i$ satisfying $P'_i\cap E_n\ \subset E^\circ_n$.

To arrange Conditions (b)~and~(c) we apply cubical small cancelation theory. Consider the cubical presentation
$\langle X|Y_i\rangle$. By Hypothesis~(iii), the complexes $Y_i$ are virtually special and have finitely many immersed codim--$2$--hyperplanes.
Since $\widetilde{Y}_i$ are superconvex, there is a uniform bound on the diameters of abstract wall-pieces. Let $K\subset \widetilde X$ be the compact subcomplex from the definition of relative cocompactness. If $j\neq i$ or $g'\neq P_i$, then $g'\widetilde Y_i\cap \widetilde Y_j\subset GK$. Since $G$ is relatively hyperbolic, the intersections $g'P_ig'^{-1}\cap P_j$ are finite and hence there is a uniform bound on the diameters of abstract cone-pieces $g'\widetilde Y_i\cap \widetilde Y_j$.

We can thus apply Lemma~\ref{lem:quasiconvex} and replace $P_i'$ by further finite index subgroups satisfying its conclusion. Condition~(c) follows directly from Lemma~\ref{lem:quasiconvex}(3). Let $\overline{A}$ be the hyperplane in $\overline{X}$ that is the projection of $\widetilde{A}$, as in Lemma~\ref{lem:quasiconvex}(1).
The group $\overline{G}$ acts cocompactly on $\overline{X}_c=\nclose{\{P_i'\}}\backslash GK$. We can assume that $\overline{X}_c$ is connected and contains an edge dual to $\overline{A}$. We will prove that $\overline{A}\cap\overline{X}_c$ is quasiconvex in $\overline{X}_c$, which means that its stabilizer $\overline{H}$ is quasiconvex in $\overline{G}$, giving Condition~(b). By Lemma~\ref{lem:quasiconvex}(2), any two points of $N(\overline{A})\cap\overline{X}_c$ are connected in $\overline{X}$
by a geodesic $\gamma$ that lies in the union of $N(\overline{A})$ and the translates of $\overline{Y}_i$ intersecting $\overline{A}$. Every component of $\gamma-\overline{X}_c$ is contained in some translate of the closure $Z_i$ of $\overline{Y}_i-\overline{X}_c$. By the last part of the definition of relative cocompactness in Section~\ref{sec:div}, the group $P_i$ acts cocompactly on $\widetilde{Y}_i\cap GK$. Thus, since $P_i'$ is of finite index in $P_i$, the intersection $Z_i\cap \overline{X}_c$ is compact. Hence we can form a quasigeodesic $\gamma_c$ by replacing in $\gamma$ each component of $\gamma-\overline{X}_c$ in a translate $g'Z_i$ by a path of uniformly bounded length in $g'\overline{Y}_i\cap \overline{X}_c$. The quasigeodesic $\gamma_c$ is contained in the union of $N(\overline{A})$ and the translates of $\overline{Y}_i\cap \overline{X}_c$ intersecting $\overline{A}$. Since $\overline{Y}_i\cap \overline{X}_c$ are uniformly bounded, $\gamma_c$ is at uniform distance from $\overline{A}$, as desired. This completes the proof of Condition~(4) of Criterion~\ref{thm:HWcriterion}.

To prove Condition~(5) we need to replace (c) by
\begin{itemize}
\item[\noindent(c$'$)]
$\overline{g}\notin\overline{H}_1\overline{H}_2,$
\end{itemize}
where $H_1,H_2$ are the stabilizers of intersecting hyperplanes $\widetilde{A}_1,\widetilde{A}_2\subset \widetilde{X}$ and $g\in G-H_1H_2$. It suffices to consider $P_i'$ provided by Lemma~\ref{lem:quotient_double_separable}. Once we have (a),~(b)~and~(c$'$) we appeal to \cite[Thm 1.1]{Min}, which says that in hyperbolic groups with separable quasiconvex subgroups, double cosets of quasiconvex subgroups are separable as well.
\end{proof}

\end{document}